\documentclass[11pt,a4paper]{article}

\usepackage[a4paper,top=3.8cm,bottom=3.8cm,left=3cm,right=3cm]{geometry}%
\usepackage{hyperref}
\usepackage[T1]{fontenc}
\usepackage{amsmath,amssymb,mathtools,amsthm,amsopn}
\usepackage{mathpazo}
\usepackage{graphicx}  
\usepackage[usenames,dvipsnames]{xcolor}

\usepackage{tocloft}
\DeclareMathOperator{\Abn}{Abn}

\renewcommand{\Im}{\operatorname{Im}}

\usepackage{fancyhdr}
\usepackage{fourier-orns}

\usepackage{enumitem}
\usepackage{cite}

\usepackage[dayofweek]{datetime}
\usepackage{mathrsfs}
\renewcommand{\phi}{\varphi}
\DeclareMathOperator{\length}{length}

\newcommand{\step}[1]{\par\medskip\noindent\it#1\rm}
\newcommand{\Scalar}[2]{\left\langle#1,#2\right\rangle}
\newcommand{\scalar}[2]{\langle#1,#2\rangle}
\newcommand{\norm}[1]{\left\Vert#1\right\Vert}

\newcommand{\abs}[1]{\lvert#1\rvert}

\renewcommand{\gg}{\mathfrak{g}}
\newcommand{\g}{\gamma}

\newcommand{\wh}{\widehat}
\renewcommand{\H}{\mathbb{H}}
\newcommand{\E}{\mathbb{E}}
\newcommand{\e}{\varepsilon}
\renewcommand{\r}{\rho}

\newcommand{\s}{\sigma}

\newcommand{\la}{\lambda}

\newcommand{\wt}{\widetilde}

\newcommand{\ol}{\overline}

\renewcommand{\d}{\delta}

\newcommand{\p}{\partial}

\DeclareMathOperator{\diag}{diag}

\newtheoremstyle{pippo}  % name of the style to be used
  {}       % measure of space to leave above the theorem. E.g.: 3pt
  {}       % measure of space to leave below the theorem. E.g.: 3pt
   {\sffamily}   % name of font to use in the body of the theorem
 {}        % measure of space to indent
  {\bfseries}  % name of head font
  {.}   % punctuation between head and body
  {1ex}       % space after theorem head
  {}           % Manually specify head
\newtheoremstyle{pluto}  {}{}
{\slshape}  {}{\bfseries}  {.} {1ex}    {}

\newtheorem{theorem}{Theorem}[section]
\newtheorem{proposition}[theorem]{Proposition}

\newtheorem{lemma}[theorem]{Lemma}

\theoremstyle{pluto}

\newtheorem{remark}[theorem]{Remark}

\usepackage{mdwlist}

\newcommand{\R}{\mathbb{R}}
\newcommand{\G}{\mathbb{G}}
\newcommand{\F}{\mathbb{F}}
\newcommand{\N}{\mathbb{N}}
\newcommand{\C}{\mathbb{C}}

\renewcommand{\d}{\delta}
\renewcommand{\t}{\tau}

\renewcommand{\a}{\alpha}
\renewcommand{\b}{\beta}

\DeclareMathOperator{\Span}{span}

\numberwithin{equation}{section}

\let\oldbibliography\thebibliography
\renewcommand{\thebibliography}[1]{%
  \oldbibliography{#1}%
  \setlength{\itemsep}{0pt}%
}

\usepackage{titlesec}
\titleformat{\section}{%
\normalfont\large\bfseries}{\thesection.}{1em}{}
\titleformat{\subsection}{%
\normalfont\normalsize\bfseries}{\thesubsection.}{1em}{}
\begin{document}

\title{On the  lack of semiconcavity of the subRiemannian distance in a class of Carnot groups
 }
\author{Annamaria Montanari \and Daniele Morbidelli}

\date{}

\maketitle

% \tableofcontents

\begin{abstract}
  We show by explicit estimates that the SubRiemannian distance in a Carnot group of step two is locally 
semiconcave away from the diagonal if and only if   the group does not contain
abnormal minimizing curves. 
% the M\'etivier condition is satisfied. 
Moreover, we prove that local semiconcavity fails to hold in the step-3 Engel group, even in the weaker ``horizontal'' sense. 
\end{abstract}

% \begin{abstract}
%    We consider a family $\H:= \{X_j = f_j\cdot\nabla: j=1,\dots, m\}$ of $C^1$ vector 
% fields in
% $\R^n$ and let $s\in\N$. We assume that for all $p\in\{1.\dots, s\}$ and $j_1, \dots, j_p\in \{1, \dots, m\}$
% the \emph{horizontal derivatives} $X_{j_1}X_{x_2}\cdots X_{j_{p-1}}f_{j_p}$
% exist and are Lipschitz continuous with respect to the control distance defined by~$\H$.
% Then we show that different notions of commutators agree.
% This involves an accurate analysis of some algebraic identities involving nested commutators which seem to have an independent interest. 
% 
% We apply these  results to  discuss the regularity properties of a class of \emph{almost exponential maps} associated with the system of vector fields.
% Ultimately, we get the proof of the doubling property and the Poincar\'e inequality for H\"ormander vector fields under an intrinsic ``horizontal regularity'' assumption on their coefficients.
% \end{abstract}
% 

\section{Introduction}
% \color{red}
% This paper contains some negative results on the problem of regularity of subRiemannian spheres.  
% In particular

It is well known that subRiemannian spheres are rather irregular objects.
Already in the simplest example---the
 Heisenberg group---the subRiemannian distance from the origin is only Lipschitz-continuous at points of the center of the group.
Furthermore, it can be shown that the only subRiemannian manifolds where (small) spheres are smooth are the Riemannian ones (see \cite{AgrachevBarilariBoscain}). 

% To state our results we introduce here some notation and  we refer to the monographs \cite{Agrachev,AgrachevBarilariBoscain,Rifford14}  for a complete discussion of the subject. 

The irregularity of the distance function is mainly governed by the presence of \emph{abnormal geodesics} (see Section~\ref{generalissimo}). Indeed, the function $d(x_0,\cdot)$ can not be smooth at any  point  $x$ connected to $x_0$ by an abnormal 
length-minimizer (see \cite{AgrachevBarilariBoscain}).
Furthermore, it has been shown in 
several papers by  Agrachev,  Bonnard,   Chyba and   Kupka \cite{AgrachevBonnard97}, Tr\'elat \cite{Trelat} and Agrachev \cite{Agrachev15} that,   
under the \emph{corank 1} assumption, where in particular all abnormal extremals are \emph{strictly abnormal}, 
at a point $x$ along an abnormal length-minimizing curve $\gamma$ 
leaving from $x_0$,
the subRiemannian sphere centered at $x_0$ is tangent to $\gamma$ in  a suitable sense and ultimately the distance  from $x_0$ can not  be expected to be even Lipschitz at~$x$. 
% Moreover, for subRiemannian structures of higher step, we can expect that the distance from a fixed point $x_0$ is only H\"older continuous at a point $x$ lying on an abnormal geodesic.

On the other side, it is known that  abnormal minimizers  do not appear at all 
for a subclass of two-step Carnot groups (M\'etivier groups) and, by a result of Chitour, Jean and Tr\'elat \cite{ChitourJeanTrelat}, in  the very large class furnished by \emph{generic} subRiemannian structures of rank at least three.

In the papers \cite{CannarsaRifford,FigalliRifford10}, Cannarsa and Rifford, and Figalli and Rifford showed that in a bracket generating subRiemannian manifold where  
  all  length-minimizing paths are strictly normal,  
   the subRiemannian distance from a fixed base point $x_0\in M$ is locally semiconcave in $M\setminus\{x_0\}$.  Since local semiconcavity implies  local Lipschitz-continuity, 
  this result can not be extended to the situation where  corank 1 abnormal minimizers appear.

However, there are subRiemannian manifolds and more specifically Carnot groups which 
do not belong to the class in \cite{CannarsaRifford,FigalliRifford10}, because they 
contain abnormal minimizing paths, but  
 do not enjoy the corank 1 assumption of 
 \cite{Trelat} and \cite{Agrachev15}, because  abnormal minimizing paths   are normal too (we say that they are \emph{normal-abnormal}). 
This class includes  
all non M\'etivier two-step Carnot groups and some step-three  Carnot groups.

In this paper we show some negative results on the local semiconcavity of subRiemannian distances in the setting  of non M\'etivier two-step groups  
and in the step-three Engel group. 
We also discuss a weaker property, namely the \emph{horizontal semiconcavity}
 and we show that, in all two-step free groups, such property holds ``pointwise"  at all  abnormal points, where the usual Euclidean notion fails to hold. We plan to come back to a detailed study of local 
 horizontal semiconcavity for the distance in two-step Carnot groups in a subsequent   work.
 On the other side, it turns out that in  the three-step Engel group the horizontal semiconcavity fails to hold.

Besides its relevant role in the optimal transport problems studied in \cite{FigalliRifford10}, local semiconcavity of the subRiemannian distance 
plays  a   role in the construction of suitable ``barrier functions'' in potential theory which are a fundamental tool in the study of second order nondivergence subelliptic PDEs with measurable coefficients (see \cite{GutierrezTournier11}, \cite{Tralli12}, 
\cite{Montanari14}).

  To state our result, we also introduce briefly some notation for two-step Carnot groups. 
Let  $(x,t)$ be coordinates in $\R^m\times\R^\ell$. Fix a family $A^1,\dots, A^\ell\in\R^{m\times m}$ of skew-symmetric matrices
and define the composition law  
\begin{equation}\label{opusdei} 
(x,t)\cdot(\xi,\tau)=\Bigl(x+\xi,t+\tau+\frac{1}{2}\scalar{x}{A\xi} \Bigr)                                                         
\end{equation} 
where  $\scalar{x}{A\xi}:=(\scalar{x}{A^1 \xi}),\dots, \scalar{ x}{A^\ell\xi})\in\R^\ell$ and $\scalar{\cdot}{\cdot}$ denotes the inner product in~$\R^m$.
 We always   assume the H\"ormander condition
$\Span\{(A_{jk}^1,\dots, A_{jk}^\ell) :1\leq j<k\leq m\}=\R^\ell$ and we denote by $d$ be the subRiemannian distance defined by the family of   left-invariant 
 vector fields  
$
 X_j=\p_{x_j}+\frac 12 \sum_{k=1}^m\sum_{\a=1}^\ell
 {A_{kj}^\a}x_k \p_{t_\a},
$ for $j=1,\dots, m$.
See Section~\ref{generalissimo}. 
% Finally, note that  $ [X_j,X_k]=A_{jk}=\sum_\a A_{jk}^\a\p_\a$ and assume the  H\"ormander condition, i.e.
% $\Span\{A_{jk}:1\leq j<k\leq m\}=\R^\ell$.

Here is our statement on two-step Carnot groups, where we   always denote by $d$ the subRiemannian distance from the origin.
\begin{theorem}\label{teoremalibero} 
Let $(\G,\cdot)=(\R^n,\cdot)=(\R^m_x\times\R^\ell_t,\cdot) $
 be the  two-step Carnot group equipped with the  law \eqref{opusdei}.  
Then, at any   $(x,0)=\gamma(1)$, final point of an abnormal minimizer $\gamma$ leaving from the origin,  there are $C>0$ and  $\tau \in\R^\ell$ such that  we have 
\begin{equation}\label{thefirst}
 d(x,\beta \tau)-d(x,0)\geq C\abs{\beta} \quad\text{ for all $\beta\in[-1,1]$.}                                                                               \end{equation} 
Moreover, if
 $(\G,\cdot)=(\R^n,\cdot)$  is    free,    then for 
any   
$(x,t)=\gamma(1) $, final point of an abnormal minimizer $\gamma$ leaving from the origin,
 there  are   $C>0$ and  $(0,\tau)\in \G$   such that 
\begin{equation}\label{liberato} 
  d(x,t+\beta \tau)- d(x,t)\geq C\abs{\beta}\quad \text{for all $\beta\in[-1,1]$.}                                                                                  \end{equation} 
\end{theorem}
 Remark that in two-step Carnot groups 
 abnormal minimizers are always normal  (see \cite[Section~20.5]{Agrachev} 
or \cite[Theorem~2.22]{Rifford14}). 
Both estimates of this theorem ensure that the distance is not semiconcave (see  the definition in   \eqref{lasemi}).

 It is known that 
for step-two Carnot groups,  $x\mapsto d(0,x) $ is Lipschitz for $x$ belonging to  compact sets  which do not intersect the origin.
Then, 
failure of semiconcavity can be visualized as a presence of an outward Lipschitz cusp on a suitable ``vertical section"  of the sphere. Inner Lipschitz cusps do not conflict with semiconcavity (think of the Heisenberg group).

Our second result concerns the  three-step Engel group $\E=\R^4$. 
In this setting  any abnormal  minimizer leaving from the origin  is 
contained in a line (\cite{Sussmann96,LiuSussmann}). The group law can be written in the form
\begin{equation}\label{enghel2} 
 x\cdot  \xi =\Bigl(x_1+\xi_1, x_2+\xi_2, x_3+\xi_3+ x_1\xi_2,
 x_4+\xi_4 +\frac{x_1^2}{2}\xi_2+x_1\xi_3\Bigr)
\end{equation}
 (see \cite[p.~285]{BonfiglioliLanconelliUguzzoni})   and the abnormal line containing the origin is  
$  \{(0,x_2,0,0)\in\R^4:x_2\in\R\} $. 
We consider the control distance associated with the left-invariant vector fields
\begin{equation*}
 X_1=\p_1 
\quad\text{and}	\quad X_2 =\p_2+x_1\p_3+\frac{x_1^2}{2} \p_4.
\end{equation*}
It follows from the results
of \cite{AgrachevBonnard97} that the distance   from the origin 
$d=d(0,\cdot)$  is not locally semiconcave at any point of such line. 
Here we  prove a further result, showing that the distance is not even semiconcave in horizontal directions in any open set intersecting the abnormal line.
  Here is our result.  
\begin{theorem}
% Consider in $\R^4 $ with coordinates $(x_1,x_2,x_3,x_4)$ the left-invariant
% vector fields 
% and let $\cdot $ be the Lie group law  in \eqref{enghel2}  (which makes $X_1,X_2$ left invariant).
For all $x_2 \in\R$ there is $C>0$ such that, if $\abs{x_4}$ is small, then
\begin{equation}\label{assetto} 
d(0,x_2,0,x_4)-d(0,x_2,0,0)\geq C\abs{x_4}.
\end{equation} 
Furthermore, we have the horizontal estimate 
\begin{equation}\label{numerodos}
  \limsup_{(y_1,y_2)\to 0}\frac{d\bigl(e^{y_1 X_1+y_2 X_2}(0,x_2,0,0)\bigr)
 -d\bigl((0,x_2,0,0)\bigr)}{y_1^2+y_2^2}=+\infty.
 \end{equation}
%  \end{theorem}
\end{theorem}
The first inequality also 
follows from the  estimate for Martinet vector fields proved in \cite{AgrachevBonnard97} (see Remark  \ref{martino1} below),
but our   proof  is more elementary. 
To the best of our knowledge,
 estimate \eqref{numerodos}   is new.

Our arguments to estimate distances are not based on exact calculations with geodesics, which in some cases  are rather
difficult  (see e.g. \cite{ArdentovSachkov11,ArdentovSachkov14}). 
We use  properties  of minimizers to localize abnormal points and we estimate the distance from the origin of close points by elementary direct arguments.   

The paper is structured as follows.
Section \ref{generalissimo} contains some general preliminaries. 
In Section \ref{passettoduetto} we discuss the step-two case and in Section 
\ref{cuba} we discuss the Engel model.

% \color{Sepia}
\section{General preliminaries}\label{generalissimo} 
\subsection{Control distances, endpoint maps and extremals}  
Let us start by recalling the vocabulary we will use in the following sections.
For a complete discussion of the subject we refer to the monographs \cite{Agrachev,AgrachevBarilariBoscain,Rifford14}.

Given   a family $X_1,\dots, X_m$ of linearly independent smooth
vector fields in $\R^n,$ the  subRiemannian distance associated with the family is defined as follows. 
An   absolutely continuous      path $\gamma\in W^{1,2}((0,1),\R^n)$ 
is said to be \emph{horizontal} if 
there is a \emph{control} $u\in L^2((0,1),\R^m)$ such that
we can write $\dot\gamma(t)=\sum_{j=1}^mu_j(t) X_j(\gamma(t)) $ for a.e.~$t\in(0,1)$. 
The subRiemannian length of a horizontal path $\gamma$ is  $\length(\gamma):=\int_0^1\abs{u(t)}dt $.
Given $x,y\in\R^n$, the subRiemannian distance between $x$ and $y$ is 
$d(x,y)=\inf\bigl\{ \int_0^1\abs{u(t)}dt\}$, where the infimum is taken among all horizontal curves $\gamma$ such that $\gamma(0)=x$ and $\gamma(1)=y$. If the H\"ormander condition holds (i.e., the vector fields, together with their commutators of sufficiently large order span a space of dimension $n$ at any point $x\in\R^n$) 
then for any pair of points $x,y\in\R^n$ there is a horizontal path connecting $x,y$ and therefore $d(x,y)$ is finite. Furthermore, it turns out that for close points, the infimum is a minimum.

Given a fixed point $x_0\in\R^n$, and given $u\in L^2((0,1),\R^m)$, we consider the a.e. solution $\gamma_u$ of the nonautonomous  	Cauchy problem
\begin{equation}\label{nonautonomo} 
\dot\gamma = \sum_j u_j(t)X_j(\gamma)\quad\text{with $\gamma_u(0)=x_0$}. \end{equation}  If $\gamma_u
\in W^{1,2}((0,1),\R^n)$ is globally defined on $[0,1]$, we define the endpoint map
$
 E(u):=\gamma_u(1).
$
In Carnot groups, it turns out that the map
$E:L^2\to \R^n$ is globally defined
and  smooth.   We say that $\gamma$ has constant speed if $\abs{u(s)}_{\R^m}=C$ for a.e.  $s\in[0,1]$.

Let $x_0\in\R^n$ be a fixed point and let $ x \in\R^n$. Assume that there is a constant-speed path $\gamma:[0,1]\to\R^n$  which is a length minimizer between $x_0$ and $x$, i.e.~$\length(\gamma)=d(x_0,x)$.  
This implies that 
there is a nonzero
vector $(\xi_0,\xi)\in\R\times \R^n$ such that
\begin{equation}\label{estremale} 
\xi_0\scalar{u}{v}_{L^2}  +\bigl\langle\xi, dE(u) v\bigr\rangle_{\R^n} =0\quad\forall \;v\in L^2=L^2((0,1),
\R^m),
\end{equation} 
where the linear map $dE(u):L^2((0,1),\R^m)\to \R^n$ denotes the differential of $E$. 
% \footnote{\color{red}Cose che non ci servono. \color{black}
% %
% Per la precisione, $\xi\in T^*_{E(u) }\R^n$. \`E un covettore nel punto finale. Allora possiamo trasportarlo in ogni punto $\gamma(t)$ con la seguente prescrizione: sia $P_t^1$ il flusso non autonomo che porta dal tempo $t$ al tempo $1$.\footnote{$P_t^s(y)$ \`e l'unica soluzione del problema 
% \[
%  \dot\gamma(s)=\sum_j u_j(s)X_j(\gamma(s)\qquad \gamma(t)=y.
% \]
% }
% A questo punto definiamo che cos'\`e il covettore $\xi(t)\in T^*_{\gamma(t)}\R^n=\R^n$ tramite la seguente formula
% \begin{equation}
%  \begin{aligned}
% \scalar{\xi}{(P_{t*}^1 X_j)(E(u)) }
% & \equiv
%  \scalar{\xi}{d P_{t}^1(\gamma(t) X_j(\gamma(t)) }=
%  \scalar{d P_{t}^1(\gamma(t))^T\xi}{X_j(\gamma(t))}
%  \\&=\scalar{\xi(t)}{X_j(\gamma(t))}.
% \end{aligned}
% \end{equation} 
% Abbiamo definito $\xi(t):=d P_{t}^1(\gamma(t))^T\xi$.
% }
If \eqref{estremale} holds, we say that $u$ is an \emph{extremal control}, or that the corresponding curve $\gamma_u$ given by \eqref{nonautonomo} is an \emph{extremal curve}.
Clearly,  it suffices to consider the case  $\xi_0=1$ and  $\xi_0=0$.
If \eqref{estremale} holds for some $(\xi_0,\xi)$ with  $\xi_0=1$, then we say that $u$ is a \emph{normal extremal control}, and $\gamma_u$ is a \emph{normal extremal curve}. If instead \eqref{estremale} holds for some $(\xi_0,\xi)$ with  $\xi_0=0$, then we say
that $u$ (resp.~$\gamma_u $) is an  \emph{abnormal} extremal control (resp.~curve). 
Equivalentely, abnormal controls are those controls $u\in L^2$ such that $dE(u):L^2\to\R^n$ is not 
open;  they  are sometimes called \emph{singular controls} and the corresponding curves are called \emph{singular curves}. 
The choice of $(\xi_0,\xi)$ is not unique, and it may happen that a control is both normal and abnormal. In such case we say that $u$ is normal-abnormal.
If $\gamma=\gamma_u$ is an abnormal curve, the set of $(\la_0,\la)\in\R\times \R^n$  such that  \eqref{estremale}
holds is a subspace whose dimension is called the \emph{corank} of $\gamma$ (see \cite{Trelat,Agrachev15}). Corank 1 extremals can not be normal-abnormal.
 Finally, a normal control/curve which is not abnormal is called \emph{strictly normal} and an abnormal control/curve which is not normal is called \emph{strictly abnormal}.
 
 It is known that all abnormal length minimizing curves in two-step Carnot groups cannot be strictly abnormal  (see \cite{Agrachev}).

\subsection{Two-step groups and M\'etivier condition}
 \label{cagatina} 
% It is well known that any two-step Carnot group with a subRiemannian metric, in suitable coordinates has this form, where $X_1,\dots, X_m$ are orthonormal.  
Let $\mathfrak{g}=V_1\oplus V_2$ be  a two-step nilpotent stratified Lie algebra (i.e.~$[V_1,V_1]=V_2$ and $[\gg,V_2]=0$). Let  $\scalar{\cdot}{\cdot}_{V_1}$ be an inner product on $V_1$. Fix an orthonormal basis $X_1,\dots, X_m$ of $V_1$ and any basis $T_1,\dots, T_\ell$ 
of $V_2$. Then we have the commutation relations
$[X_j,X_k]=\sum_{\a=1}^\ell A_{jk}^\a T_\a$ for suitable constants $A_{jk}^\a=-A_{kj}^\alpha\in\R$. Since $\operatorname{Exp}:\gg\to\G$ is a global diffeomorphism, we  can identify the Lie group $\G=\operatorname{Exp}(\gg)$ with $\R^m\times\R^\ell$ via exponential coordinates of the first kind
\begin{equation}
 \label{espoprimo}
\R^m\times\R^\ell\ni (x_1,\dots, x_m,t_1,\dots, t_\ell)\simeq \operatorname{Exp}\Bigl(\sum_j x_jX_j+\sum_\a t_\a T_\a\Bigr)\in\G=\exp(\gg)
\end{equation} 
Finally, an application of the Baker--Campbell--Hausdorff--Dynkin formula (see \cite{BonfiglioliLanconelliUguzzoni}) shows that the group law in $\G$  in the coordinates $(x,t)\in\R^m\times\R^\ell$ takes the form 
\begin{equation}\label{opusdei2}  
(x,t)\cdot(y,s)=\Bigl(x+y,t+s+\frac{1}{2}\scalar{x}{Ay} \Bigr)                                                         \end{equation} 
mentioned in \eqref{opusdei}. 
 A subRiemannian frame of orthonormal horizontal left-invariant 
 vector fields in given by  
$
 X_j=\p_{x_j}+\frac 12 \sum_{k=1}^m\sum_{\a=1}^\ell
 {A_{kj}^\a}x_k \p_{t_\a},
$ for $j=1,\dots, m$.
Moreover, $[X_j,X_k]=A_{jk}=\sum_\a A_{jk}^\a\p_\a$. We assume the H\"ormander condition $\Span\{A_{jk}:1\leq j<k\leq m\}=\R^\ell$.

In a two-step group, given $\eta\in V_2^*$, define $J_\eta:V_1\to V_1$ by the formula
 $\scalar{J_\eta X}{X'}=\eta([X,X'])$. 
 We say that the group satisfies the M\'etivier condition  
\cite{Metivier80} if the linear map $J_\eta$ is an isomorphism   for all $\eta\in V_2^*\setminus\{  0\}$.  
 The  M\'etivier class  includes the class of the groups of Heisenberg type  (with strict inclusion, see \cite[Section~7]{MullerSeeger04} or  \cite{BonfiglioliLanconelliUguzzoni}).
 An equivalent way to state the M\'etivier condition is by requiring that the map 
 $
  \R^m \ni y\mapsto \scalar{Aw}{y}\in\R^\ell$  is onto for all $w\in\R^m\setminus \{0\}$.
   Another equivalent assumption is that  the square matrix $\sigma A:=\sum_{\a=1}^\ell \sigma_\alpha A^\alpha 
   \in\R^{m\times   m}$ is nonsingular for all $\sigma=(\sigma_1,\dots, \sigma_\ell)\neq 0\in\R^\ell$.

\subsection{Semiconcavity} Following \cite[Definition~1.1.1]{CannarsaSinestrari} and \cite{FigalliRifford10}, we say that a continuous function $f:\Omega\to\R$
is \emph{semiconcave} on the open set $\Omega\subset\R^n$ if there is $C>0$ such that
\begin{equation}\label{lasemi} 
 f(x+h)+f(x-h)-2f(x)\leq 2C \abs{h}^2,
\end{equation}
for all $x,h\in\R^n$ such that the segment $[x-h,x+h]$ is contained in $ \Omega$.
Equivalently, there is $C>0$ so that
\begin{equation*}
 \la f(y)+(1-\la)f(x)-f(\la y+(1-\la )x)\leq C\la(1-\la)\abs{x-y}^2
\end{equation*}
for all $x,y$ such that $[x,y]\subset\Omega$ and $\la\in[0,1]$.
Roughly speaking, second order derivatives of a semiconcave function can be $-\infty$, but they must be bounded from  
above by some positive constant $C<\infty$. See \cite[Chapter~2]{CannarsaSinestrari}.

The following theorem has been shown by Cannarsa and Rifford \cite{CannarsaRifford}, and Figalli and Rifford \cite{FigalliRifford10}:
\begin{theorem}
 Let $M$ be a subRiemannian manifold with subRiemannian distance $d.$   Let  $x_0\in M$ and assume that for all $y\in M$
 every length minimizing path connecting $x_0$ and $y$ is nonsingular. Then, the distance function $y\mapsto d(x_0,y)$ is locally semiconcave on $M\setminus \{x_0\}.$ 
\end{theorem}

\section{Step-two groups}\label{passettoduetto} 
\subsection{Some (mostly known) facts on  step-two groups }

\subsubsection{Endpoint map and extremal paths}

% \color{Sepia}
Let $\R^m\times\R^\ell$ be equipped with the group law \eqref{opusdei2}.
Denote by $e=(0,0)$ the  identity  element of the group  
and by $d(x,t)$ the distance from the origin of  $(x,t)\in\R^m\times\R^\ell$.
 The ODE for the curve $\gamma =(x,t)$ associated with a control $u\in L^2((0,1),\R^m)$ is
\begin{equation}\label{cervello} 
 \dot x(s)=u(s)\qquad \dot t(s)=\frac 12\scalar{x(s)}{A u(s)},
 \qquad \text{with $(x(0),t(0))=(0,0)$}
\end{equation}
where $\scalar{x}{Au}=(\scalar{x}{A^1u}, \dots, \scalar{x}{A^\ell u})$.
Given $u\in L^2(0,1)$, the endpoint  map $E(u)=\gamma(1)=(x(1), t(1))$ has the form  
\[
 E(u)=\Bigl(\int_0^1 u(s)ds,\;
 \frac 12\int_0^1\Bigl\langle\int_0^s u,Au(s)\Bigr\rangle ds\Bigr).
\]
As calculated in \cite{AgrachevGentileLerario}, its differential $dE(u):L^2\to\R^m\times\R^\ell$
has the following form
\begin{equation}
\label{endozzi} 
\begin{aligned}
 dE(u)v &=\Bigl(\int_0^1v,  \frac 12 \int_0^1\Bigl\{
 \Big\langle  \int_0^s u, Av(s)
 \Big\rangle+
  \Big\langle \int_0^s v,A u(s)\Big\rangle\Bigr\} ds\Bigr)
\\&
   =\Bigl(\int_0^1v,\; \int_0^1\Big\langle     A\Bigl(\frac x2 -\int_0^s u \Bigr), v(s)\Big\rangle
   ds
   \Bigr).
\end{aligned}
\end{equation}
We integrated by parts and we let  $\int_0^1 u=x$.

Next, we recapitulate the discussion in \cite{AgrachevGentileLerario}. 
Let $u\in L^2(0,1)$ be a minimizing control for the problem 
$\min\{\norm{u}^2_{L^2(0,1)}: E(u)=(x,t)\}$. 
Since minimizing controls in step-two Canot groups are always 
normal (this follows from the second order analysis of the Goh condition, see \cite[Section~20.5]{Agrachev} or \cite[Theorem~2.22]{Rifford14}), there is  a nontrivial (co)vector $(\xi,\tau)\in\R^m\times \R^{ \ell}$  such that  
\[
\begin{aligned}
0& =   \scalar{u }{v }_{L^2} - \scalar{(\xi,\tau)}{dE(u)v}
\\&
   = \int_0^1 \scalar{u(s)}{v(s)}ds
   - \int_0^1\scalar{\xi}{v(s)}ds
   -\int_0^1\Big\langle     \tau A\Bigl(\frac   x2-\int_0^s u \Bigr), v(s)\Big\rangle
   \Bigr)
    \quad \text{ for all $v\in L^2(0,1)$.}
\end{aligned}
\]
Here $\t A:=\sum_{\a=1}^\ell \t_\a A^\a =-(\t A)^T \in \R^{m\times m}$. 
Since $v\in L^2$ is arbitrary, we get
\begin{equation}\label{nonno} 
u(s)=\xi +\tau A \frac{x}{2}-\tau A\int_0^s u(\r)d\r\quad\text{for all $s\in[0,1]$}.
\end{equation} 
Therefore, $\dot u(s)=-\tau A u(s)$ and then, according to \cite[Proposition~5]{AgrachevGentileLerario},
\begin{equation}\label{joas} 
 u(s)=e^{-\tau As }u,
\end{equation} 
for a suitable $u\in\R^m.$ 
It is easy to recognize that, since $A$ is skew symmetric, then
 $e^{-\t A s}\in O(m)$ is an orthogonal $m\times m$ matrix. 
 Therefore, the path $\gamma$ has constant speed and $\length(\gamma)=\abs{u}$.
%  The corresponding solution becomes
% \begin{equation}
% % \label{curso} 
% \gamma(s)= (x(s),t(s))=\Bigl(
%  \int_0^s 
%  e^{-\t A\r}  u 
%  d\r
%  ,\frac 12 \int_0^s \Scalar{\int_0^\r e^{-\t A\beta} d\beta  \;  u}{A e^{-\t A\r}  u}
%  d\r
%  \Bigr)
% \end{equation} 
% It may happen that $\tau=0$ and this means that $u(s)=\xi$ is a constant control.

Let $u \in L^2(0,1)$ be an abnormal extremal. Then  by definition there is $(\eta,\sigma)\in\R^m\times\R^\ell\setminus\{(0,0)\}$ such that 
\[
\begin{aligned}
 0 & =\scalar{(\eta,\sigma)}{dE(u) v} =\int_0^1
 \Big\langle 
 \eta+\sigma A \Bigl(\frac{x}{2}-\int_0^s u\Bigr), v(s)  ds\Big\rangle
 \quad \text{ for all $v\in L^2(0,1)$.}
\end{aligned}
\]
Since  $v $ is arbitrary,  one gets
\begin{equation}\label{jerso} 
 \eta+\sigma A \frac{x}{2}-\sigma A\int_0^s u=0\quad \text{for all $s\in[0,1]$,}                                                                                \end{equation} 
with the usual convention $\sigma A:=\sum_{\a=1}^\ell \sigma_\a A^\a$.
Note that it must be $\sigma\neq 0$. Otherwise $(\eta, \sigma)$ becomes trivial. 
Differentiating we obtain, according with \cite[Lemma~2.4]{Kishimoto} and \cite{hsu} the condition 
\begin{equation}\label{cervellino} 
\sigma A u(s)=0\quad\text{ for almost all  $s$}
\end{equation} (which implies 
$\eta=0$).
% \color{red}
Since $\ker(\sigma A)$ is a subspace of dimension at most $m-2$, the structure of the ODE 
\eqref{cervello} implies that, letting $\Abn(e)=\{\gamma(1):\gamma \text{ is abnormal and
} \gamma(0)=e\}$  
we have  
\begin{equation}\label{giornale} 
\Abn(e)   
\subseteq \bigcup \Bigl\{\G_W: W\text{ subspace of }\R^m,\quad \dim W\leq m-2
\Bigr\}       \end{equation}  
where $\G_W$ is the subgroup 
\begin{equation}
\label{givudoppio} 
\G_W    :=\Span\big\{\bigl(w, \scalar{w'}{Aw''} \bigr): w,w',w''\in W\big\}, \end{equation} 
% \times\wedge^2 W$, where
which is a Carnot group of step $ r  \in\{1,2\}$.    To check this claim, note that 
\eqref{cervellino} ensures that there is a subspace $W\subset \R^m$ of dimension at most $m-2$ such that $u(s)\in W$   a.e.~in $s\in[0,1]$. Then $x(s)=\int_0^s u\in W$ and 
$
t(s)=\frac 12\int_0^s\scalar{x(\rho)}{Au(\rho)}d\rho\in\Span\{\scalar{w'}{A w''}:w', w''\in W\}.
$  
 The inclusion    \eqref{giornale} can be strict, but it  is an equality for free groups (see \cite{LeDonneMontgomeryOttazziPansuVittone} and Remark \ref{rangolo} below).

Furthermore,  \eqref{cervellino} implies that a control of the form  $u(s)=e^{-\tau A s}u$ is  abnormal if and only if 
there is $\sigma\in\R^\ell\setminus \{0\}$  such that
\begin{equation}
\sigma A(\tau A)^m u=0\qquad\text{for all $m\in\N\cup\{0\}$}.
\end{equation} 
It may happen that $\sigma\in\Span\{\tau\}$ and 
in such case, comparing \eqref{nonno} and \eqref{jerso}, we see that $u(s)=e^{-\tau As}u=u
\in\operatorname{ker} \tau A
$ is a constant control.
%%
% \begin{subequations}
%  \begin{align}
% \label{unino}   &u(s)= e^{-\t A s}\bar u         
%   \\ \label{duetto}  &
%   \bar u = \xi + \tau A\frac{x}{2}
%   \\ \label{tros} 
%  &\eta=0
%  \\  \label{quas} &\sigma A(\tau A)^m \bar u =0\quad\forall \;m=0,1,2,\dots.
%           \end{align}
% \end{subequations}

% \color{MidnightBlue}

\subsubsection{Bivectors and skew-symmetric matrices}  
 If we denote by $e_1,\dots, e_m$ the canonical basis of $\R^m$, we define $\wedge^2\R^m:=\Span\{e_j\wedge e_k:1\leq j<k\leq m\}$. 
Given  two vectors $x,y\in\R^m$, the elementary bivector  $z=x\wedge y \in\wedge^2\R^m$  can be expanded as
\[
 x\wedge y=\sum_{j }(x_j e_j)\wedge\sum_k  (y_k e_k)= \sum_{1\leq j<k\leq m}(x_jy_k-x_ky_j)e_j\wedge e_k
 =:\sum_{1\leq j<k\leq m}(x\wedge y)_{jk}e_j\wedge e_k.
\]
On $\wedge^2\R^m$ we define the standard inner product on elementary bivectors
\[
 \scalar{x\wedge y}{\xi\wedge \eta}=\scalar{x}{\xi}\scalar{y}{\eta}
 -\scalar{x}{\eta}\scalar{y}{\xi}
 \quad\text{for all } x,y,\xi,\eta\in\R^m.
\]
This is equivalent to the requirement that the family 
$e_j\wedge e_k,$ with $1\leq j<k\leq m,$ is orthonormal in $\wedge^2\R^m$. 
The inner product $\scalar{z}{\zeta}$ can be extended by linearity to general bivectors $z=\sum_{a=1}^n x_a \wedge y_a $ and $\zeta=\sum_{\a=1}^\nu \xi_\a\wedge \eta_\a$, for any $x_a,y_a,\xi_\a,\eta_\a\in\R^m$.
Note that if $\R^m=V\oplus W$ decomposes as a sum with $V\perp W$ and  we  choose orthonormal 
bases 
$v_1,\dots, v_p$ of $V$
and $w_1,\dots w_q$ of $W$, it turns out that  the family  $\{ v_j\wedge v_k, v_j\wedge w_\a,w_\a\wedge w_\beta: 1\leq j<k\leq p,\quad  1\leq \a<\beta\leq q\}$ is an orthonormal  basis of   $\wedge^2\R^m$ and ultimately
the three terms in the decomposition 
\begin{equation}
\label{joeyy} 
 \wedge^2\R^m=\wedge^2 V\oplus (V\wedge W)\oplus\wedge^2W 
\end{equation} 
are pairwise orthogonal. Here and hereafter we are keeping the short notation $V\wedge W:=\Span\{v\wedge
w:v\in V,w\in W\}$.

Let $M=-M^T\in\R^{m\times m}$ be a  skew-symmetric matrix of rank $2p\leq m$. By spectral theory, there are $p$ two-dimensional pairwise orthogonal subspaces $V_1,\dots V_p$,
$p$ positive numbers $\la_1,\dots, \la_p>0$  and a corresponding orthonormal basis $v_h,v_h^\perp $ of each $V_h$ such that 
\[
 Mv_h=  \la_h v_h^\perp\quad\text{and }\quad M v_h^\perp = -\la_h v_h
 \quad\text{ for all $h=1,\dots,p$.}
\]
In other words, we can write
$
 Mx=\sum_{h=1}^p \la_h \bigl(\scalar{x}{v_h} v_h^\perp - \scalar{x}{v_h^\perp}v_h\bigr)
$.
Observe that $\Im M=\oplus_{h=1}^p V_h$  and $\operatorname{ker} M=(\oplus_h V_h)^\perp$.  
It may happen that $\la_i=\la_j$ for some $i\neq j$.  The generic element of $M$ is $M_{jk}=\Bigl(\sum_{h=1}^p\la_hv_h^\perp\wedge v_h\Bigr)_{jk}$. 
The rank 
of the bivector $\sum_{h=1}^p \la_h v_h^\perp \wedge v_h\in\wedge^2\R^m$ is by definition~$p$.
Moreover, the space $\Span\{v_h,v_h^\perp: 1\leq h\leq p\}$ is called the  \emph{support} of the bivector.

A short computation shows that the exponential of $M$ applied to $x\in\R^m$ is
\begin{equation}
\begin{aligned}
 e^Mx &=\sum_{h=1}^p \cos(\la_h) \bigl(\scalar{x}{v_h}v_h+\scalar{x}{v_h^\perp}v_h^\perp\bigr)
 +\sum_{h=1}^p\sin(\la_h)(\scalar{x}{v_h} v_h^\perp - \scalar{x}{v_h^\perp}v_h)
 \\&
 \qquad
 +\Bigl\{x-\sum_{h=1}^p \bigl(\scalar{x}{v_h}v_h+\scalar{x}{v_h^\perp}v_h^\perp\bigr)
 \Bigr\}.
\end{aligned}
\label{duenove} 
\end{equation} 
\subsubsection{Extremal curves in free groups}
 Let  $ \F_m\equiv \F_{m,2}:=\R^m\times \wedge^2\R^m$ with the group law
 \begin{equation}\label{liberone}
(x,t)\cdot(\xi,\tau)=\Bigl(x+\xi, t+\tau+\frac 12 x\wedge  \xi\Bigr).                                                               \end{equation} 
 Here for convenience of notation we used $\wedge^2\R^m$ instead of $\R^\ell$
 and we made the choice of matrices $A^{jk}\in\R^{m\times m}$ defined as follows: 
 $
  A^{jk}x =x_k e_j- x_j e_k 
 $. Then,  for any $x,\xi\in\R^m$ we indicate with $\scalar{x}{A\xi}\in\wedge^2\R^m$  the bivector
 \begin{equation}
\begin{aligned}
 \scalar{x}{A\xi}=x\wedge\xi=\sum_{1\leq  j<k\leq m}(x\wedge \xi)_{jk}e_j\wedge e_k=
 \sum_{1\leq  j<k\leq m}(x_j\xi_k-x_k\xi_j)e_{j}\wedge e_k.
\end{aligned}
 \end{equation} 
 Let $u(s)=e^{-\t A s}u$ be a normal extremal control. 
   Since $-\tau A$   is a  
 skew-symmetric matrix, there  are $p\leq \frac{n}{2}$,  strictly
 positive 
 numbers $\la_1,\dots, \la_p> 0$ and corresponding pairwise orthogonal vectors  
 $a_1, a_1^\perp,\dots, a_p, a_p^\perp,z$ such that 
 \begin{equation}\label{controllone} 
  \begin{aligned}
u(s)& =\sum_{k=1}^p \bigl(\cos(\la_k s) a_k+\sin(\la_k s) a_k^\perp
\bigr)+ z,         \end{aligned}
 \end{equation} 
 where $\abs{a_k}=\abs{a_k^\perp}>0$ for all $k=1,\dots, p$ and $z\perp\Span\{a_k, a_k^\perp:1\leq k\leq p\}$. Here it may be  $z=0$. 
% The fact that the group is free 
The  free-group assumption
ensures that the matrix $-\tau A$ can be any skew-symmetric matrix and thus any control $u$ of the form 
\eqref{controllone} is a normal extremal control.
% in~$\R^m\times\wedge^2\R^m$. 

 Moreover, we may assume without loss of generality that  in \eqref{controllone} the following ``nondegeneration condition'' holds
 \begin{equation}
  \label{divero}
% \left\{\begin{aligned}
% &  \lambda_j\neq\la_k\quad\text{for all}\quad j \neq k         
% \\&
0<\la_j\neq \la_k \quad\text{for all}\quad j\neq k.
% \end{aligned}
% \right.
 \end{equation} 
Otherwise,   if $\la_j=\la_k$ for some $j\neq k$, then
we can write 
 \[\cos(\la_j s) a_j+\sin(\la_j s)a_j^\perp +
 \cos(\la_j s) a_k+\sin(\la_j s)a_k^\perp=
 \cos(\la_j s)( a_j+a_k)+
 \sin(\la_j s)(a_j^\perp+a_k^\perp).
  \]
Observe that if we add to   condition \eqref{divero} the requirement $\la_j<\la_k$ if $j<k$, then all the data $p, \la_k, a_k, a_k^\perp, z$
  are uniquely determined by $u(s)$. 
Finally, the length of the curve~$\gamma_u$ corresponding to the control \eqref{controllone} is 
$\length(\gamma_u)^2=\abs{z}^2+\sum_{k=1}^p \abs{a_k}^2$.
The curve corresponding to the extremal control \eqref{controllone} lives in the subgroup $W\times\wedge^2 W$, where
\begin{equation}\label{doppiovu} 
 W:=\Span\{a_1,a_1^\perp,\dots, a_p, a_p^\perp,z\}.                                                   \end{equation} 
The discussion below shows that $\gamma_u$ is nonsingular in the subgroup $\G_W:=W\times\wedge^2 W$. In general the inclusion $\G_W\subset dE(u)L^2$ is strict. 
% \footnote{For example, in $\F^3=\R^3\times\wedge^2 \R^2$,  the curve associated to the control $u(s)=\cos(\la s) e_1+\sin(\la s) e_2\in W= \Span\{e_1, e_2\}\subset \R^3$ belongs to the subgroup $W\times\wedge^2 W$, but $dE(u)L^2=\R^3\times\wedge^2 \R^3$. 
% } 

In order to characterize singular extremals, we will use the following linear algebra lemma.
\begin{lemma}\label{vanderm} 
 Let $v_1,\dots, v_p\in\R^m$ and let $0<\la_1<\la_2<\cdots<\la_p$ be positive numbers. Then, 
 \begin{equation*}
 \begin{aligned}
\Span\{& 
\la_1^{2k-1} v_1+\la_2^{2k-1} v_2+\cdots +\la_p^{2k-1}v_p:1\leq k\leq p \} 
  \\&
  =
  \Span\{ \la_1^{2k-1} v_1+\la_2^{2k-1} v_2+\cdots +\la_p^{2k-1}v_p:k\in\N\}
  \\&=\Span\{v_1,v_2,\dots,v_p\}.
  \end{aligned}
\end{equation*} 
An analogous statement holds changing the powers $2k-1$ with $2k$.
\end{lemma}
\begin{proof}
 In both equalities $\subseteq$ is trivial. To accomplish the proof, it suffices to show that
 the set in the first line contains $\Span\{v_1,v_2,\dots,v_p\}$. To see this fact  
 observe that
 \[
\begin{aligned}
 [\la_1v_1+\cdots +\la_pv_p\mid
 &
 \la_1^3v_1+\cdots+\la_p^3 v_p \mid \cdots|
\la_1^{2p-1}v_1+\cdots+\la_p^{2p-1} v_p ]
\\&
=
[v_1|\cdots| v_p]
\begin{bmatrix}
\la_1 &\la_1^3 & \cdots \la_1^{2p-1}
\\
\la_2 &\la_2^3 & \cdots \la_2^{2p-1}
\\
\dots &\dots &\cdots \dots
\\
\la_p &\la_p^3 & \cdots \la_p^{2p-1}
\end{bmatrix}
\end{aligned}
 \]
The thesis follows because the Vandermonde matrix is nonsingular.   
\end{proof}

Next we recall the characterization of singular  extremal controls (see also \cite{LeDonneMontgomeryOttazziPansuVittone}). 
Let $u$ be a normal extremal control of the form
\eqref{controllone}  satisfying the nondegeneration condition \eqref{divero}.   Then, $u$ is singular if and only if  there is 
a nontrivial skew-symmetric matrix $\s A \in\R^{m\times m}$ such that $\sigma A u(s)=0$ for all $s$. By properties of the kernel of skew-symmetric matrices this 
is equivalent  
to say  
that there is a $(m-2)$-dimensional
subspace $W\subseteq\R^m$ such that 
$u(s)\in  W$ for all $s$. Equivalently, $\dim \Span\{u^{(k)}(0):k\in\N\cup\{0\}\}\leq  m-2$, which means
\begin{equation}\label{libertas} 
 \dim\Span\Bigl\{ z+\sum_{k=1}^p a_k,
\sum_{k=1}^p\la_k a_k^\perp,\sum_{k=1}^p \la_k^2 a_k,\sum_{k=1}^p \la_k^3 a_k^\perp,\dots 
 \Bigr\}\leq m-2.
\end{equation} Since we assume \eqref{divero}, using the lemma above, it is easy to recognize that this is equivalent to the requirement
\begin{equation}
 \label{jersey}
 \dim\Span\{a_1,a_1^\perp, \dots, a_p,a_p^\perp,z\}\leq m-2
\end{equation} 
% there is a $(m-2)$-dimensional subspace $V$ such that

\begin{remark}\label{rangolo} Formula \eqref{jersey} is related with the parametrization of the abnormal set provided in 
% In the papers  \cite{LeDonneLeonardiMontiVittone} and \cite{LeDonneMontgomeryOttazziPansuVittone} the abnormal set in free two-step Carnot groups is characterized in the following  elegant way 
% (see 
formula~(3.9) in \cite{LeDonneMontgomeryOttazziPansuVittone}. Indeed it implies that
\begin{equation}\label{pansu} 
\Abn^{\textup{nor}}(e)=  \operatorname{Abn}(e)= \bigcup\Bigl\{W\times \wedge^2 W: W\subset\R^m\quad \dim W =  m-2\Bigr\}
\end{equation}  
where $\Abn^{\textup{nor}}(e)$ indicates the endpoints of normal-abnormal curves leaving from the origin.
The first $\subseteq$ inclusion is obvious and the second follows from \eqref{giornale}. The fact that
$\Abn^{\textup{nor}}(e)$ contains the union on the right-hand side can be seen as follows.
Let $W\subseteq\R^m$ be a subspace of dimension $\dim W=m-2$. Then $W\times\wedge^2 W$ is isomorphic to the free two-step group  with $m-2$ generators. Therefore, for each point $(w,\xi)\in W\times\wedge^2 W$ there is a control of the form \eqref{controllone} with $a_1,a_1^\perp, \dots, a_p, a_p^\perp,z\in W$ and such that the curve $\gamma$ arising from such control  connects the origin with $(w,\xi)$.
% \footnote{Parte cancellata
% Our formula \eqref{jersey} contains essentially the same information.
% Indeed,
% the $\subseteq$ inclusion follows 
% to make a brief  informal comparison, note that 
% Our previous discussion shows that $\Abn^{\textup{nor}}(e)$ is the set of  final points $
%  (x(1), t(1))$ of curves $s\mapsto (x(s), t(s))$ corresponding to controls of the form \eqref{controllone} with   assumption \eqref{divero}. Such curves have the form 
% \[
% \left\{\begin{aligned}
%  &x(s)=zs+\sum_{k=1}^p\Bigl\{
%  \frac{\sin(\la_ks)}{\la_k}a_k+
%  \frac{1-\cos(\la_k(s)}{\la_k}a_k^\perp
%  \Bigr\},\qquad 
%  \\&
%  t(s)=\frac 12\int_0^s
%  x(\rho)\wedge u(\rho) d\rho
%  \in\Span\Bigl\{a_h\wedge a_k, a_h\wedge a_k^\perp, z\wedge a_h, z\wedge a_h^\perp :
% 1\leq h,k\leq p\Bigr\}.
% \end{aligned}\right.
% \]
% We do not need here the calculation of $t(s)$. 
% Observe that such curve belongs to the subgroup $W\times\wedge^2 W$, where $W= \Span\{z,a_h,a_h^\perp:1\leq h\leq p\}$ has dimension at most $m-2$, by \eqref{jersey}.
% This shows the inclusion $\subseteq$ in 
%  \eqref{pansu}. 
%  }
\end{remark}

% \begin{example}
% For example, in $\F_{3,2}=\R^3\times\wedge^2\R^3$,  
% according to \cite{LeDonneMontgomeryOttazziPansuVittone} and to our condition \eqref{jersey}, singular extremal are constant: $u(s)=\xi$ for $s\in[0,1]$.
% Thus, we have
% $
%  \Abn(\F_{3,2}) \operatorname{Abn}(\mathbb{F})=\R^3\times \{0\}$.
% In $\F_{4,2}$ we have instead $\Abn\F_{4,2}=\bigcup W\times\wedge^2 W $
 % \label{esempiazzo}
% $ \mathbb{F}=\R^3\times\wedge^2\R^3$
% \[
%  (x_1,x_2,x_3,t_{12},t_{13},t_{23})\cdot(y_1,y_2,y_3,s_{12},s_{13},s_{23})
% =(x,t)\cdot (y,s):=\Bigl(x+y, t+s+\frac 12 x\wedge y\Bigr)
% \]
% where $ x\wedge y:= (x_1y_2-x_2y_1)e_{12}+ (x_1y_3-x_3y_1)e_{13}+(x_2y_3-x_3y_2)e_{23}.
% $
% \footnote{
% Forse le matrici non ci servono pi\'u.
% \[
%  A_1=\begin{bsmallmatrix}
% 0 &1&0
% \\
% -1&0&0
% \\ 0&0&0
%    \end{bsmallmatrix}
%    \qquad
%     A_2=\begin{bsmallmatrix}
% 0 &0&1
% \\
% 0&0&0 
% \\ -1&0&0
%    \end{bsmallmatrix}
% \qquad 
% A_3=
% \begin{bsmallmatrix}
%  0 &0 &0
%  \\
%  0&0& 1
%  \\
%  0 & -1& 0
% \end{bsmallmatrix}
% \quad\Rightarrow \quad\tau A=
% \begin{bsmallmatrix}
%  0 & -\t_1 &- \t_2
%  \\
%  \t_1 & 0 & -\t_3
%  \\
%  \t_2&\t_3 & 0 
% \end{bsmallmatrix}
% \]
% }
%  In $\F_{4,2}$
% \end{example}

\subsubsection{Extremals in general two-step groups}
If $(\R^m\times\R^\ell,\cdot)$ is a 
  two-step  Carnot group  
 with law \eqref{opusdei}, normal extremal curves can be described similarly to the free case, but there are some differences.  Indeed, given an extremal control
$u(s)=e^{-\t A s}u$, while  in the free case $-\t A$ was the most general skew-symmetric matrix, 
here, as observed by \cite{AgrachevGentileLerario}, the matrix $-\tau A$ should belong to the  
subspace of $\mathfrak{so}(m),$
generated by $A_1,\dots, A_\ell$. 
Anyway, applying spectral theory to the matrix $-\tau A$, we see that $u(s)$ can be written in the form 
 \begin{equation}\label{controllone2} 
  \begin{aligned}
u(s)& =\sum_{k=1}^p \cos(\la_k s) a_k+\sin(\la_k s) a_k^\perp + z,         \end{aligned}
 \end{equation} 
 where, as in the free case 
 we assume without loss of generality 
 the nondegeneration condition
 \begin{equation}\label{controllone3} 
%   \label{divero}
% \left\{\begin{aligned}
% &  \lambda_j\neq\la_k\quad\text{for all}\quad j \neq k         
% \\&
0<\la_j\neq \la_k \quad\text{for all}\quad j\neq k.
% \end{aligned}
% \right.
 \end{equation} 
Again, making the further requirement $\la_j<\la_k$ if $j<k$, then all the data $p, \la_k, a_k, a_k^\perp, z$
  are uniquely determined by $u(s)$. 
If we let, as in the free case 
$ W:=\Span\{a_1,a_1^\perp,\dots, a_p, a_p^\perp,z\}$, then  by \eqref{cervello},  it 
turns out that the curve corresponding to the extremal control \eqref{controllone2} lives in the subgroup $\G_W$ introduced in~\eqref{givudoppio}.

The description of singular extremals is less precise than in the free case. However, 
by \eqref{cervellino} and Lemma \ref{vanderm}, we   can say that a control of the form \eqref{controllone2} under the nondegeneration condition is singular if and only if there is $\sigma\in\R^\ell$ such that the associated subspace~$W$ satisfies
% \begin{equation*}
 $W\subset \ker(\sigma A).$
% \end{equation*} 
Equivalently, there is $\sigma\neq 0$ such that 
$\sigma\perp \scalar{Aw}{y}$ in $\R^\ell$ for all $w\in W$ and $y\in\R^m$. This ensures that $W\subset\R^m$ has dimension at most $m-2$.
Furthermore, under \eqref{controllone3}, it turns out that 
$\gamma_u$ is nonsingular in the subgroup $\G_W $ defined in \eqref{givudoppio} (if it would be singular, then  $\{u(s):s\in[0,1]\}$  would be contained in a subspace of dimension at most $ \operatorname{dim}W-2$).

\begin{remark}\label{confronto} 
% If we use exponential coordinates of the first kind to identify with $\R^m\times\R^\ell$  a two-step nilpotend stratified  group,  we see  that 
The 
objects of the discussion above have a strict 
relation  with the abnormal varieties $Z^\lambda$ studied in  \cite{LeDonneLeonardiMontiVittone} and \cite[Section~3.1]{LeDonneMontgomeryOttazziPansuVittone}. Indeed, fixed the basis $X_1,\dots, X_m, T_1,\dots T_\ell$ of $\gg=V_1\oplus V_2$
as in  Section \ref{cagatina} and the  dual basis $\eta_1,\dots,\eta_m,\theta_1,\dots, \theta_\ell$ of $\gg^*=V^*_1\oplus V^*_2$, then, choosing the  covector 
$\lambda=\sum_\a \sigma_\a\theta_\a\in V_2^*$, a computation shows that, in the exponential coordinates \eqref{espoprimo} 
\[
 Z^\la
%  =Z^{(0,\sigma)}
 = \{(x,t)\in\R^m\times\R^\ell:\sigma A x=0\}=\ker(\sigma A)\times\R^\ell.
\]
where $\sigma A=\sum_\a\sigma_\a A^\a$ as usual. 
Thus, $\mathfrak{z}^\la \cap V_1 =\{ \sum_j x_jX_j:x\in \ker \sigma A\}\subset V_1$ and 
% where $\mathfrak{z}^\la$ appears in  \cite[eq.~(3.1)]{LeDonneMontgomeryOttazziPansuVittone}.
% Therefore, 
\begin{equation*}
% \label{sottogr} 
H^\lambda =\Span\bigl\{
(x,
\scalar{\xi}{A\eta}): x,\xi,\eta\in 
\ker(\sigma A)
\bigr\}                     \end{equation*} 
is the subgroup appearing in \cite[Eq.~(3.1)]{LeDonneMontgomeryOttazziPansuVittone}.
% it is shown 
\end{remark}

Next we calculate the image of the differential of the endpoint map at extremal controls
in terms of the associated subspace~$W$. 
 \begin{proposition}\label{dimensa} 
Let $u\in L^2(0,1)$ be a normal extremal control  
of the form 
\eqref{controllone2} satisfying the nondegeneration condition \eqref{controllone3}. Then,  if 
$W=\Span\{a_1,a_1^\perp,\dots, a_p, a_p^\perp,z \}$, we have
  \begin{equation*}
   \Im dE(u)=\Span \{(\xi, \scalar{Aw}{\eta}): w\in W\; \xi,\eta\in\R^m\}.
  \end{equation*} 
%   \color{red} ?? \color{black}
 \end{proposition}
\begin{proof}
Formula \eqref{endozzi} immediately implies $\subseteq$.

To see $\supseteq$, we test formula \eqref{endozzi} against    sequences of smooth functions approximating the $\d$ function and its derivatives of order $\ell\geq 1$. Precisely, let $\phi\in C_c^\infty(\left]-1,1\right[)$ be a nonnegative averaging kernel with $\int_{-1}^1\phi(s) ds=1$.  Then define  the family $(\phi_n)_{n\geq 2}$, by  $\phi_n(s):=n\phi(ns-1) $. It turns out that $\phi_n\in C_c^\infty(\mathopen]0,1\mathclose[)$ and $\phi_n$
is an approximation of the Dirac mass at $s=0$ as $n\to \infty$.
Moreover, 
for  $\ell=0, 1,2,\dots$, $\xi\in\R^m$ and $n\in\N$, the family $(\phi_n^\ell)_{n\geq 2}$, 
$\phi_n^\ell(s):=(\frac{d}{ds})^\ell\phi_n(s) $ approximates the $\ell$-th derivative of the Dirac mass, as $n\to \infty$.

Let us take $\xi\in\R^m$ and define  $v_n^\ell(s)=\phi_n^\ell(s) \xi$.
  Testing \eqref{endozzi} with $(v_n^0)_{n\in\N}$ and passing to the limit as $n\to\infty$ we find
  \[
\Im    dE(u)\supseteq  \Big\{\bigl(\xi,\Big\langle A\frac{x}{2},\xi\Big \rangle\bigr):\xi\in\R^m\Big\}.
  \]
If instead $\ell\geq 1$, calculating $dE(u)v_n^\ell$ and letting $n\to \infty$, we find
\[
 \Im    dE(u)\supseteq  \Big\{\bigl(0,-\scalar{A x^{(\ell)}(0)}{\xi}) :\xi\in\R^m\Big\}\quad \text{for all $\ell=1,2,\dots$.}
\]
The proof is easily concluded because $\Span \{x^{(\ell)}(0):\ell\geq 1\}=W$.
%   Note that   It is well known that  $v_n^\ell(s)=\phi_n^\ell(s)$
\footnote{Notice that $L_W:= \{ \scalar{Aw}{\eta}:w\in W,\; \eta\in\R^m \}$ in general is not a subspace of 
$ \R^\ell$.   This for instance happens if $\R^m\times\R^\ell=\R^4\times\wedge^2\R^4$, $\scalar{x}{Ay}=x\wedge y$ and $W=\Span\{e_1,e_2\}$. 
   In such case, $e_1\wedge e_3$ and $e_2\wedge e_4\in
     L_W   $, 
   but 
$e_1\wedge e_3 + e_2\wedge e_4\notin L_W$.
}
\end{proof}

\begin{remark}
 In the nonfree case, it is not true that $\Abn(e)$ can be parametrized as $\bigcup \{\G_W 
 :\dim(W)\leq m-2\}$, where $\G_W$ is the subgroup in \eqref{givudoppio}. A counterexample is given by
 a direct product $\H_{x_1,x_2,t}\times \R_{x_3}$ of the Heisenberg group with the Euclidean line.  
Here, for $x,\xi\in\R^3$ we define
$\scalar{x}{A\xi}=x_1\xi_2-x_2\xi_1$
In this case,  for any $w=(w_1,w_2,0)\in\R^3\setminus \{(0,0,0)\}$, a curve of the form $\gamma(s)=(sw_1, sw_2,0,0)$ is an extremal and is contained in the subgroup
$\G_W$ where $W= \Span\{w\}$ is one-dimensional. However, $\gamma$ is nonsingular in the product.
\end{remark}

Using Proposition \ref{dimensa} 
it is easy to see that abnormal minimizing curves  appear if and only if 
the M\'etivier condition fails. This statement is implicitly contained in \cite[Eq.~3.2)]{LeDonneMontgomeryOttazziPansuVittone}.
\begin{proposition}
 Let $\G=\R^m\times\R^\ell$ be the group in \eqref{opusdei}. Then there exists a   nontrivial
   abnormal length minimizing path if and only if the M\'etivier condition fails.
\end{proposition}
\begin{proof}
 Let $u\in L^2(0,1)$ be a   nonzero   abnormal length minimizing control. Since $u$ must be normal-abnormal, it has the form \eqref{controllone2} and we may assume the nondegeneration \eqref{controllone3}. Applying Proposition \ref{dimensa}, we see that if 
 $   0\neq   w\in W$, then the dimension of $\Span\{\scalar{Aw}{\eta}:\eta\in\R^m\}$ must be strictly less than $\ell$. This means that the M\'etivier condition fails.
 
 On the other side, if the M\'etivier condition fails, let $w\in\R^m\setminus\{0\}$ be such that $ \eta\mapsto \scalar{Aw}{\eta}$ is not onto from $\R^m$ to $\R^\ell$. Then, by Proposition \ref{dimensa}, we see that the curve $\gamma(s)=(sw,0)$ is an abnormal minimizer. 
\end{proof}

\subsection{Failure of semiconcavity in two-step Carnot groups}
\subsubsection{Free groups}
We show  estimate \eqref{liberato} of Theorem \ref{teoremalibero}.
% 
% that given the nilpotent two-step free group $\F_{m,2}$  defined in  \eqref{liberone}, the distance is not locally semiconcave  if $m\geq 3$. More precisely, we show that the semiconcavity estimate for the distance from the origin fails at any  $(x,t)\in\Abn(\F_{m,2})=\Abn^{\textup{nor}}(\F_{m,2})$.
% \begin{theorem}
 Namely, given  $(  x,  t)=\gamma(1)$, final point of an abnormal minimizer~$\gamma$,
 we want to show that  there is $ \s \in\wedge^2\R^m$ such that
\begin{equation}\label{giovannino} 
 \liminf_{\beta\to 0}\frac{d(  x,  t+  \beta   \s )- d(  x,  t)}{\abs{\beta}}>0                                                                            \end{equation} 
% Note that \eqref{giovannino} implies that 
%  \begin{equation*}
%   \lim_{\beta=0}\frac{d(  x,  t+  \beta   \s )+
%   d(  x,  t-  \beta   \s )- 2 d(  x,  t)}{\beta^2}=+\infty
%  \end{equation*} 
% and then the distance is not semiconcave.
\begin{proof}[Proof of \eqref{giovannino}]
 Let $(  x,   t)=\gamma(1)=(x(1),t(1))$, where $\gamma$ is   a normal-abnormal extremal. This means that $\gamma$ originates from a control of the  form $
u(s) =\sum_{k=1}^p \cos(\la_k s) a_k+\sin(\la_k s) a_k^\perp + z,   $ 
 where  as usual we assume that $0<\lambda_j<\la_k$  for all $ j < k$
and moreover we have  the singularity condition
\begin{equation*} 
\dim \Span\{a_1,a_1^\perp, \dots, a_p, a_p^\perp,z\} \leq m-2,
\end{equation*} 
(here $z$ may  possibly vanish).
Let $\Span\{a_1,a_1^\perp, \dots, a_p, a_p^\perp,z\}=:W.$
% \color{red} 
Let
% \[
% v_1,\dots, v_q\quad\text{of}\quad 
$V:=W^\perp=
\Span\{a_1,a_1^\perp, \dots, a_p, a_p^\perp,z\}^\perp$.
% \]
% \color{black} 
The singularity condition ensures that 
% \color{red} 
$\dim V  \geq 2$. 
Let $\F_V: = V\times\wedge^2 V$ be the subgroup generated by $V\times\{0\}$. 
We claim that for any nonzero   bivector   $\s \in\wedge^2 V$,
we have
\begin{equation}\label{key}
 d\big(  x,  t+\beta \s \bigr)\geq d\big(  x,  t\bigr)+C\abs{\beta}.
%  \qquad\text{for small $|\beta|$}.
\end{equation} 

To prove the claim,  fix $\sigma\in\wedge^2 V\setminus \{0\}$ and let  $\beta\in\R$. Take a minimizing control $u\in L^2(0,1)$ and let  $\gamma=(x,t):[0,1]\to\F_{m,2}$ be the corresponding minimizing path joining $(0,0)$ and $(x,t+\beta\sigma)$. Assume also the constant speed condition
\[
\abs{u(s)}= \abs{\dot x(s)}=d\big(  x,  t+\beta \s \bigr)\quad\forall s\in[0,1].
\]
Decompose orthogonally 
\begin{equation}\label{unodiuno} 
\begin{aligned}
u(s)&=: u_V(s)+ u_W(s)\in V\oplus W
\\
x(s)&=\int_0^s u=
% =\sum_{j=1}^q x_j(s) v_j+\sum_{\mu\leq p}
%  \{\xi_\mu(s) a_j+\eta_\mu(s) b_j\} +\zeta(s) z
%  \\&
 :x_V(s)+ x_W(s) 
 \in V\oplus  W
 .
\end{aligned}
\end{equation}  
% Correspondingly write $\dot  x(s) =
%   :\dot x_V(s)+ \dot x_W(s)
%  \in V\oplus W
% $.
% \begin{equation*}
% \begin{aligned}
% \dot  x(s)&=
% % \sum_{j=1}^q x_j(s) v_j+\sum_{\mu\leq p}
% %  \{\xi_\mu(s) a_j+\eta_\mu(s) b_j\} +\zeta(s) z
% %  \\&
%  =:\dot x_*(s)+ \dot x_0(s)
%  \in V^\perp\oplus V
%  .
% \end{aligned}
% \end{equation*}  
% Next
Thus,
\begin{equation}\label{duedidue} 
\begin{aligned}
 t(s)&=\frac 12\int_0^s (x_V +x_W)\wedge (u_V+u_W)
 \\&
   = \frac 12\int_0^s x_V \wedge  u_V 
   +
   \frac 12\int_0^s  x_V \wedge u_W 
   +
   \frac 12\int_0^s x_W\wedge u_V
   +\frac 12 \int_0^s x_W\wedge u_W
   \\&
   =: t_V(s)+ t_{*}(s)+ t_W(s) 
%    \\&
   \in\wedge^2 V \oplus (V\wedge W)\oplus  \wedge^2 W,
\end{aligned}
\end{equation}
where we let \[t_V(s)= 
\frac 12\int_0^s x_V \wedge  u_V\in\wedge^2V
\quad\text{and}\quad
t_W(s)= \frac 12 \int_0^s x_W\wedge u_W\in\wedge^2 W.
\] 
 By \eqref{joeyy}, the three     terms in the last sum are pairwise orthogonal. 
The path  $s\mapsto \gamma_V(s)=(x_V(s),t_V(s)) \in V\times \wedge^2 V$ is admissible in the Carnot group $V\times \wedge^2V$  and the path $\gamma_W$ is admissible in $W\times\wedge^2 W$.

Next we look at the final point of $\gamma_V$. 
Since \[ W
\ni   x=x(1)=x_V(1)+x_W(1) \in 
V\oplus W \]
we have 
$x_V(1)=0$ and $x_W(1) =  x$.
Moreover, since
\[
  \wedge^2 W \oplus\wedge^2
  V\ni   t+ \beta \s =t(1)=t_V(1)+
  t_* (1)+t_W(1)\in \wedge^2 V\oplus(V\wedge W)\oplus
  \wedge^2 W, 
\]
it must be   $t_*(1)=0\in V\wedge W$,
$t_V(1)= \beta \s\in \wedge^2 V$ and $t_W(1)=  t\in \wedge^2W$.

Ultimately, since the path $\gamma_V$ connects the origin with $(0,\beta \s )\in \F_V=V\times\wedge^2 V$,
we have
\[
 \int_0^1\abs{u_V} \geq d_{V\times\wedge^2 V}(0,\beta \s)
 \geq C\abs{\beta}^{1/2}.
\]
Moreover, since $\gamma_W $ connects the origin with $(  x,  t)\in \F_W$,
we have
\begin{equation}
 \int_0^1\abs{u_W}\geq d_{\F_W}(  x,  t)\geq  d_{\F_{V\oplus  W}}(  x,  t)=d(x,t).                                                                                                \end{equation} 
By the constant-speed assumption $\abs{u(s)}=d(  x,  t+\beta \s)$ for all $s$, 
\[\begin{aligned}
 d(  x,   t+\beta \s)^2
 &=\int_0^1\abs{u}^2=\int_0^1\abs{u_V}^2+\int_0^1\abs{u_W}^2
 \\
 &\geq\Bigl(\int_0^1\abs{u_V}\Bigr)^2 
 + \Bigl(\int_0^1\abs{u_W}\Bigr)^2
 \\&\geq C\abs{\beta}+d(  x,  t)^2.
\end{aligned}
\]
To conclude the proof, observe that if $(x,t)=(0,0)$, then $d(0,\b\s)\geq C\abs{\b}^{1/2}\geq C\abs{\b}$ for  $\abs{\b}<1$. If instead $(x,t)\neq (0,0)$, then 
\[
\begin{aligned}
 d(x, t+\b\s) &\geq d (x,t)\Bigl(1+\frac{C\abs{\b}}{d(x,t)^2}\Bigr)^{1/2}
 \geq d(x,t)+\frac{C'}{d(x,t)}\abs{\b},
\end{aligned}
\]
for small $\abs{\beta}$.
This proves \eqref{key}.
\end{proof}

\subsubsection{General two-step groups}
Here we prove  estimate \eqref{thefirst}, which shows that local semiconcavity fails for all two-step Carnot groups at  abnormal points  of the form $(w,0)\in\R^m\times\R^\ell$.
The case of a general abnormal point  seems to be technically more complicated and we do not discuss it. A  procedure   of lifting to a free group can be  useful to discuss some specific examples, but the general case seems to require a deeper understanding of two-step Carnot group. 
% \begin{theorem}
%  Let $\R^m\times\R^\ell$ be equipped with the group law  in
%  \eqref{opusdei}. Then for any abnormal point  of the form $(w,0)\in\R^m\times\R^\ell$ there is $\sigma\in\R^\ell$ 
%   such that  
%   \begin{equation}\label{dosingo} 
%  \liminf_{\beta\to 0}\frac{d(w,\beta\sigma)-d(w,0)}{\abs{\beta}} 
%  \gvertneqq 0.
% \end{equation} 
% \end{theorem}
% 
% 

\begin{proof}[Proof of \eqref{thefirst}] Let $w\in\R^m$ be a unit vector such that the map
$y\mapsto \scalar{Aw}{y}$ is not onto from $\R^m$  to   $\R^\ell$. We claim that estimate \eqref{thefirst} holds for any  vector $\s\in\R^\ell\setminus\{0\}$  
such that 
\begin{equation}\label{eperpe} 
\scalar{Aw}{y}\perp \s\quad \text{in $\R^\ell$}\quad\forall\; y\in\R^m.
\end{equation} 
Assume without loss of generality that $\abs{\sigma}=1$  in $\R^\ell$.
Let $V:=\Span\{w\}^\perp=:W^\perp$ and
\begin{equation*}
 \G_V:=\Span\{(v ,\scalar{A v' }{v'' }): v,v',v''\in V\}.
\end{equation*}

We claim that there is $C>0$ such that $d(w,\beta \s)\geq 1+
C\abs{\beta}$ uniformly in  $\beta\in[-1,1] $.
To show the claim, let $ \gamma=(x ,t):[0,1]\to \G $ be a length minimizing constant-speed path, i.e. $\abs{  u(s)}= d(w,\beta \s) $ 
for all $s\in[0,1]$.
Decompose
\begin{equation}\label{giacomino} 
\begin{aligned}
 u(s)&= u_V(s)+ u_W(s)\in V\oplus W
 \\
 x(s)&
%  =\sum_{j=1}^{m-1}x_j(s) v_j+\zeta(s) w
 =\int_0^s u=:x_V(s)+x_W(s) \in V\oplus W.
% \\ 
%  \dot x(s)& =\sum_{j=1}^{m-1}\dot x_j(s) v_j+\dot \zeta(s) w=:\dot x_V(s)+\dot x_W\in V\oplus W,
\end{aligned}
\end{equation}
% 
% with obvious definitions $x_j(s)=\scalar{x(s)}{e_j}$ and $\zeta(s):=\scalar{x(s)}{w}$
Thus,
\begin{equation}\label{giacometto} 
\begin{aligned}
 t(s) & = \frac 12\int_0^s\scalar{x } {Au } 
%  \\& 
%  = \frac 12\int_0^s\Big\langle x_V(\rho)+x_W(\rho) ,A\bigl(\dot x_V(\rho) +\dot x_W(\rho) \bigr) \Big\rangle d\rho
 \\&
    =\frac 12\int_0^s\Scalar{x_V } {A  u_V  } 
%     \\&
    +  \frac 12\int_0^s(\scalar{x_V  } {A u_W  }+
    \scalar{x_W}{A u_V}) 
    +  \frac 12\int_0^s\scalar{ x_W } {Au_W   } 
\\&
=:t_V(s)+(t(s)-t_V(s)).
\end{aligned}
\end{equation}
where we put $
t_V(s):=\frac 12\int_0^s\Scalar{x_V } {A u_V  } $. Note that  the curve 
$\gamma_V(s)=(x_V(s),t_V(s))$ is an admissible curve 
in $\G_V$. The decomposition \eqref{giacomino} proves that $x_V(1)=0$.
Formula \eqref{giacometto} and the orthogonality condition \eqref{eperpe} tell that $\beta \s
\perp t(1)-t_V(1)$. Therefore, the required equality $t(1)=\beta \sigma$
implies that 
\[
\abs{ t_V(1)}^2= \Bigl|\beta \s-( t(1)-t_V(1))\Bigr|^2=\beta^2+
\abs{t(1)-t_V(1)}^2\geq \beta^2,
\]
because $\abs{\s}=1$.
Standard properties of two-step groups give
\begin{equation*}
% \label{collecto}
\length_{\G_V}(\gamma_V)=\int_0^1\abs{u_V(s)}ds 
\geq C\abs{\beta}^{1/2}.\end{equation*} 
A second obvious estimate concerns the curve
 $\zeta(s):=\scalar{x(s)}{w}$.
Since it satisfies $\zeta(0)=0$ and $\zeta(1)=1$, we have  
\begin{equation*}
% \label{cortarella}
\int_0^1\abs{u_W }=\int_0^1\abs{\dot \zeta}\geq 1.
\end{equation*} 

To conclude the argument, starting from the constant speed property of $\gamma=(x,t)$ and using Cauchy-Schwarz, we find
\begin{equation*}
\begin{aligned}
 d (w,\beta \s)^2
 &
 =\int_0^1\abs{u}^2=
 \int_0^1\abs{u_W}^2+\int_0^1\abs{u_V}^2
 \\&
    \geq \Bigl(\int_0^1\abs{u_W}\Bigr)^2+
    \Bigl(\int_0^1\abs{u_V}\Bigr)^2
    \\&
    \geq 1+C\abs{\beta}
    =d(w,0)^2+C\abs{\beta}
 \end{aligned}
\end{equation*} 
and the proof is concluded.
\end{proof}

% \color{magenta}

\subsection{Horizontal semiconcavity estimates at abnormal points in free groups}
By the results in \cite{CannarsaRifford,FigalliRifford10}, in a small neighborhood of the final point   $\gamma(1)=(x,t) $ of a strictly normal minimizer, the distance from the origin is semiconcave. This estimate fails if $\gamma $ is    abnormal. 
However,  a horizontal version of the semiconcavity property  persists 
at abnormal points, at least in free groups.
% \color{black}
Indeed, if  $\F_m$ is the free two-step group with $m$ generators,
  for all $(x,t)=\gamma(1)$, where $\gamma$ is abnormal length-minimizing on $[0,1]$, $\gamma(0)=(0,0)$ and
  $d(x,t)=1$   there are 
positive constants $C$ and $\delta$ so that 
% and $(x,t)\in \Abn \F_m$
% is a point on the subRiemannian  unit sphere, 
\begin{equation}\label{orizzontale} 
 \sup_{\substack{y\in\R^m,
 \; \abs{y}\leq \delta }}\frac{d\big(e^{y\cdot X}(x,t)\big)+ 
 d\big(e^{-y\cdot X}(x,t) \big)-2d(x,t)}{\abs{y}^2}\leq C.
%  \quad\forall\; (x,t)\in\Abn(\F_m)\quad d(x,t)=1.
\end{equation} 
         Here $y\cdot X:=\sum_{j=1}^m y_j X_j $ and $e^{y\cdot X}(x,t)$ denotes the value at time $t=1$ of the integral curve of $y\cdot X$ leaving from $(x,t)$. 
We do not know whether or not such estimate holds uniformly in $(x,t)$ on the unit sphere. We plan to come back to such problem in a further paper.

Estimate \eqref{orizzontale} can be proved by an induction argument and the discussion below is devoted to   the proof of such statement. 

\step{Step 1.}   Let us start by observing that if $  w_1,\dots, w_d $ is an 
orthonormal  basis of a $d$-dimensional subspace $W\subset\R^m$ and   $\G_{W}:=W\times\wedge^2 W$ is a free subgroup of $\F_m:=\R^m\times\wedge^2\R^m$, then for any point $(x,t)\in \G_W$ we have the estimate 
\begin{equation}\label{uguaglio} 
 d_{\F_m}(x,t)= d_{\G_W}(x,t)=d_{\F_d}(\xi,\tau)                                                \end{equation} 
where in the last equality we  denoted $\xi_j=\scalar{x}{w_j} $
and $\tau_{jk}=\scalar{w_j\wedge w_k}{t}$ for $j=1,\dots, d$.
The $\leq$ in the first equality of \eqref{uguaglio}  follows from the fact that $\G_W$ is a subgroup of $\F_m$. The $\geq $ holds because
\begin{enumerate}[nosep,label=(\alph*)]
 \item   If $u\in L^2((0,1),\R^m) $ is a control   in $\F_m$ such that the curve $\gamma_u$
  connects the origin with $(x,t)\in\G_W\subset \F_m$, then, the orthogonal projection
   $u_W\in L^2((0,1),W)$
 is admissible in $\F_W$ and the corresponding curve $ \gamma_W$ connects the origin with $(x,t)\in\G_W$.
 
 \item $\length_{\G_W}( \gamma_W)\leq \length_{\F_m}(\gamma)$.
\end{enumerate}
Note that the  $\geq$ inequality 
  in \eqref{uguaglio}  
may fail if  we change $\F_m$ with a  nonfree two-step Carnot group $\G$.   
 This can be seen by considering the group $\R^5
\times\R=  \G$ with operation
\[
 (x_1,x_2.x_3.x_4 ,t)\cdot(\xi_1,\xi_2,\xi_3,\xi_4,t)
 =\Bigl(x+\xi,t+\tau +\frac 12(x_1\xi_2-x_2\xi_1)+
 \frac{\a}{2}(x_3\xi_4-x_4\xi_3)\Bigr)
\]
with $\a>1$ and its subgroup $\G_W:=\G_{\Span\{e_1,e_2  \}} =\{(x_1,x_2,0,0, t)\}$. Here it turns out that 
 $
 d_{ \G}( 0,0,0,0,\beta)=\sqrt{4\pi \abs{\beta} /\a} <
 \sqrt{4\pi \abs{\beta}}=
d_{\G_W}( 0,0,0,0,\beta) 
$.  
% Lie algebra  $\Span\{X_1,X_2,X_3,X_4,T\}$ with $X_{12}=T$ and $X_{34}=\lambda T$....

\step{Step 2.}
Let us look at estimate 
\eqref{orizzontale} for  $m=3$. In such case abnormal points in the unit sphere are of the form~$(x,t)=(w,0)$ for some $w\in\R^3$ with unit norm. Then, 
any vector $y\in\R^3$ can be written in the form $y=\xi w+\eta v$, where $\xi,\eta\in\R$ and $v\perp w$ is a suitable unit vector. 
Therefore, we have 
\[ 
\begin{aligned}
  e^{y\cdot X} (w,0)   &=(w,0)\cdot\operatorname{Exp}(y\cdot X) =(w,0)\cdot(y,0)
 \\& = (w,0)\cdot (\xi w+\eta v,0 )\in \G_{\Span\{w,v\}}=\{(\xi w+\eta v,\tau w\wedge v):(\xi,\eta,\tau)\in\R^3\}
\end{aligned}\]
(here $\operatorname{Exp}$ denotes the standard Exponential map,   see \cite[Definition~1.2.25]{BonfiglioliLanconelliUguzzoni}).
Thus all points involved in the estimate  belong to a subgroup which is isomorphic to the Heisenberg group $(\H_1,\circ)$.  
Therefore we have 
\[\begin{aligned}
d&\big((w,0)\cdot(\xi w+\eta v,0)\big)+d\big((w,0)\cdot(-\xi w-\eta v,0)\big)  -2
d (w,0) 
\\&=
d_{\H_1}((1,0,0)\circ(\xi,\eta,0))+ d_{\H_1}((1,0,0)\circ(-\xi,-\eta,0))
-2 d_{\H_1} (1,0,0) \leq C(\xi^2+\eta^2),                                         \end{aligned}
\]
by the local semiconcavity of the distance in the Heisenberg group (\cite{CannarsaRifford,FigalliRifford10}).  
Since this estimate is uniform as $v\in\R^3$ is a unit vector
orthogonal to $w$, the statement in $\F_3$ follows easily.

\step{Step 3.} Next we describe the induction step. Assume that the estimate holds for $\F_{m-1}$ and let us look at 
$(x,t)=\gamma(1) \in \F_m
% \Abn (\F_m)=\Abn^{\textup{nor}} (\F_m)
$ with $d(x,t)=1$,  $\gamma(0)=(0,0)$ where $\gamma$ is an abnormal length-minimizer. 
Let $W\subset\R^m$ be the associated subspace 
introduced in \eqref{doppiovu} and assume that 
$ w_1,\dots, w_d $ is an orthonormal basis of $W$. The singularity condition means that $d\leq m-2$. Moreover, any vector $y\in\R^m$ can be written in the form 
$y=\sum_{j=1}^d\xi_j w_j+\eta v$, 
where $v\perp W$ is a suitable unit vector depending on $y$ (but we will get estimates which are uniform in $v\perp W$, $\abs{v}=1$). Therefore, we have
\begin{equation*}
         \begin{aligned}
 (x,t)\cdot\operatorname{Exp} (y\cdot X)& =(x,t)\cdot (y,0)
 \\&
 =\Bigl(\sum_{j=1}^d x_j w_j, \sum_{j<k\leq d} t_{jk}w_j\wedge w_k \Bigr)\cdot
 \Bigl(\sum_{j\leq d}\xi_j w_j+\eta v, 0\Bigr)\in \G_{W\oplus\Span\{v\}}.
\end{aligned}\end{equation*}
Thus, all involved points belong to a free subgroup which isomorphic to $\F_{d+1}$. 
  If there is an abnormal length minimizer in such subgroup that connects the origin and 
$(x,t)$, then,   
since    $d+1\leq m-1$, using \emph{Step 1} and arguing as in \emph{Step 2}, we  get the required statement~\eqref{orizzontale}. 
Otherwise, if any minimizer is normal,  we can use \cite{CannarsaRifford} or \cite{FigalliRifford10}  and we get again the desired estimate~\eqref{orizzontale}.

% \newpage

\section{Lack of semiconcavity for the control distance in the Engel group}
\label{cuba} 

Let us consider the vector fields
\begin{equation*}
 X_1=\p_1 
\quad\text{and}	\quad X_2 =\p_2+x_1\p_3+\frac{x_1^2}{2} \p_4.
\end{equation*}
It can be checked that $X_1$ and $X_2$ are left invariant on the Lie group in $\mathbb{E}=\R^4$ defined by the following law
\begin{equation}\label{enghel} 
 x\cdot  \xi =\Bigl(x_1+\xi_1, x_2+\xi_2, x_3+\xi_3+ x_1\xi_2,
 x_4+\xi_4 +\frac{x_1^2}{2}\xi_2+x_1\xi_3\Bigr)
\end{equation}
which is usally called \emph{Engel group}. See \cite[p.~285]{BonfiglioliLanconelliUguzzoni}). Such vector fields belong to the model studied in the seminal paper \cite{Sussmann96} on abnormal geodesics for rank two distributions and it is known that $\Abn\mathbb{E}=\{(0, x_2,0,0):x_2\in\R\}$.
Therefore, by \cite{CannarsaRifford} and  
  \cite{FigalliRifford10},  we know that the distance  from the origin $d$ is locally semiconcave on  $\R^4\setminus \R e_2$.  
Here we show that $d$ is not semiconcave at any point of the abnormal line.
Moreover, we show that $d$  is not semiconcave at such points even in the weaker horizontal sense.

In the papers   \cite{ArdentovSachkov11,ArdentovSachkov14} and  \cite{AdamsTie13} 
the explicit form of geodesics is established.
In principle, our estimates could be obtained as a consequences of the mentioned results. However the form of such geodesics is rather involved and working with their 
explicit equations  seems to be a  rather difficult task.

Observe
that the subset $\{(x_1,x_2,0,x_4)\}\subset\E$ with the induced vector fields $X_1=\p_1$ and $X_2=\p_2
+\frac{x_1^2}{2}\p_4$ can be identified with the Martinet subRiemannian system.  See the discussion in the following Remarks \ref{martino1} and \ref{martino2}.

Preliminarily we show that  taken the constant control $\wt u(t)=(0,1)$ for  $t\in[0,1]$, so that $E(\wt u)=(0,1,0,0)$,
we have  
\begin{equation}
\label{immo} 
 \operatorname{Im}  d  E(\wt u)=\Span\{e_1,e_2,e_3\}.                                                     \end{equation} 
We briefly check \eqref{immo}, by means of  standard formula for the differential of the endpoint map. Following the notation in \cite{AgrachevBarilariBoscain}, given $u\in L^2$, 
we denote by  $P_s^t(x)$ the solution of $\frac{d}{dt} P_s^t(x)=\sum_j
u_j(t)X_j( P_s^t(x))$, with $ P_s^s(x)=x$. Thus, we have the well known formula
% $P_t:=P_0^t$.
\begin{equation*}
\begin{aligned}
 dE(  u)v & =\int \big\{ v_1(t) dP_t^1(P_0^t (0))X_1(P_0^t(0))+v_2(t) dP_t^1(P_0^t (0))X_2(P_0^t(0))
 \}dt.
\end{aligned}
\end{equation*} 
See \cite{Montgomery,Rifford14,AgrachevBarilariBoscain}.
At the point $u=\wt u$, we have  
$
 P_0^t x=e^{tX_2}x=\Bigl(x_1 , x_2+t,
 x_3+ tx_1  , x_4+\frac{x_1^2}{2} t\Bigr),
$
so that  
\[
 dP_t^1(x)=
 \begin{bsmallmatrix}
 1&0 & 0& 0
 \\
 0 & 1 & 0 & 0 &
 \\
  1-t  & 0 & 1 & 0
 \\
 (1-t)x_1 & 0 & 0& 1
           \end{bsmallmatrix}
\quad \qquad 
P_0^t(0)=\begin{bsmallmatrix}
         0 \\t\\0\\0
        \end{bsmallmatrix}
\quad\Rightarrow\quad
dP_t^1(P_0^t(0))=
 \begin{bsmallmatrix}
 1&0 & 0& 0
 \\
 0 & 1 & 0 & 0 &
 \\
1-t
 %  \frac{1-t}{2}
 & 0 & 1 & 0
 \\
 0 & 0 & 0& 1
           \end{bsmallmatrix}.
\]
Therefore, 
\begin{equation}\label{joao} 
\begin{aligned}
 dE(u)v=\Bigl(\int_0^1 v_1(t) dt, 
 \int_0^1 v_2(t) dt,\int_0^1 (1-t )v_1(t) dt,0\Bigr)
\end{aligned}
\end{equation} 
which implies \eqref{immo}.  
The curve $x(t)=te_2 $ is both normal and abnormal. It is abnormal because   $dE(u)$ is not open. It is normal because  the equality 
\[
\la_0 \scalar{u}{v}_{L^2} 
+\big\langle (\la_1,\la_2,\la_3,\la_4),dE(u)v\big\rangle_{\R^4}=0 \quad\forall \; v\in L^2=: L^2((0,1),\R^2)
\]
holds under the choice  $\la_0=-\la_2$ and $\la_1=\la_3=0$.

It is very easy to show the failure of semiconcavity looking at the behavior of the distance in the orthogonal of $\Im dE(\wt u)$, i.e. in $\Span\{e_4\}$. 
 This is shown by estimate \eqref{assetto}, which we are now going to prove.  
A more precise version of the following proposition  can be obtained as a consequence of \cite{AgrachevBonnard97}
(see the remark after the proof).

%
%
% \begin{proposition}[]
% The distance from the origin in the Engel group \eqref{enghel} is not locally semiconcave in any open set containing   a point of the abnormal  line  $\Abn(\E)=\{(0,x_2,0,0):x_2\in\R\}$. 
% \end{proposition}
\begin{proof}[Proof of \eqref{assetto}]
It suffices to show that there is $C_0>0$ such that
\begin{equation}
 \label{jaor}
 d(0,1,0,\la)-1\geq C_0\abs{\la} \quad\text{ for all $\lambda$ close to $0$.}
\end{equation} 
% (Sharpeness of the estimate for $\la>0$ and $\la<0$ is commenten in the remark)
% This implies that 
% $
% \lim_{\la\to 0+}  \frac{1}{\la^2}\bigl(d(0,1,0,\la)+d(0,1,0,-\la)-2d(0,1,0,0) \bigr)=+\infty
% $, ending the proof.

 To show  estimate \eqref{jaor}, let us consider the control problem 
$
 \dot\gamma=u_1(t)X_1(\gamma)+u_2(t)X_2(\gamma)$ with $\gamma(0)=(0,0,0,0)$ and $\gamma(1)=(0,1,0,\lambda )$,
where   $u=(u_1,u_2)\in   L^2(0,1)$.
Note that writing $\gamma=(\gamma_1,\gamma_2,\gamma_3,\gamma_4)$, we have 
% that 
% $
%  \dot\gamma_3= \gamma_2\dot\gamma_1
%  .
% $
% Thus 
\begin{equation*}
 \gamma_3(1)= \int_{\gamma_0}  x_2dx_1 =0\quad\dot\gamma_4(1)=\frac{1}{2}\int_{\gamma_0}x_1^2 dx_2=\lambda,
\end{equation*}
where we denoted  $\gamma_0:=(\gamma_1,\gamma_2)$.
% \begin{equation*}
%  
% \end{equation*}
% 
% 
% ha la forma
% \[
%  \gamma_3(t)=\Gamma_1(t)-\frac 12 t\gamma_1(t)
% \]
% dove $\Gamma_1(t)=\int_0^t\gamma_1$.
% Siccome vogliamo raggiungere il punto $(0,1,0,x_4)$ otteniamo il vincolo
% \begin{equation}\label{prox} 
%  \gamma_3(1)=\Gamma_1(1)-\gamma_1(1)=\int_0^1\gamma_1(t)d t=0 
% \end{equation}
% Dunque
% Furthermore, since $\dot\gamma_4= \frac{\gamma_1^2}{2}\dot\gamma_2$
% \[
% \begin{aligned}
%  \dot\gamma_4
% & =-\dot\gamma_1\Bigl(\frac{1}{12}\gamma_1\gamma_2+\frac 12 \gamma_3\Bigr)
%  +\frac{1}{12}\dot\gamma_2\gamma_1^2
% \end{aligned}
% \]
% Quindi 
% \[\begin{aligned}   \gamma_4(1) & =  -\frac{1}{12}\int \dot\gamma_1\gamma_1\gamma_2 
%    -\frac 12 \int\dot\gamma_1\gamma_3
%    +\frac{1}{12}\int \gamma_1^2\dot\gamma_2
%   =\cdots
% % \\&
%    = \frac 12\int_\gamma x_1^2 dx_2,
% %     +\frac{1}{12}\int \gamma_1^2\dot \g_2
% % \\&=
% %  \frac{1}{24}\int \gamma_1^2\dot\gamma_2+\frac 12\int \gamma_1\dot \gamma_3
% %  \frac 16 t\gamma_1(t)\dot\gamma_1(t)-\frac 12\dot\gamma_1(t)\Gamma_1(t)+\frac{1}{12}\gamma_1(t)^2
% %  \\&=\frac{1}{12}\frac{d}{dt}(t\gamma_1(t)^2)-\frac 12\dot\gamma_1(t)\Gamma_1(t)
% \end{aligned}
% \]
% dove abbiamo  usato la forma di $\dot\gamma_3$ e fatto un po'  di integrazioni per parti .
% 
% Prendiamo nota dei vincoli:
% \begin{equation}\label{terzi} 
%  \int_\gamma x_1 dx_2-x_2 dx_1=0                                \end{equation} 
% e
% we have the constraint

Let $(\gamma^\la)_{\la\in\R}$ be a family of curves $\gamma^\la:[0,1]\to \R^2$ 
satisfying $\gamma^\la(0)=(0,0)$, $\gamma^\la(1)=(0,1)$ and 
% \begin{equation}
% \label{quarti} 
$\int_{\gamma^\la} x_1^2 dx_2=2\lambda$.  
% \end{equation} 
% then 
% \[\lim_{\la\to 0}
% \frac{\length_{\R^2}(\gamma^\la)-1}{\abs{\la} } \gvertneqq 0
% \]
% This implies \eqref{jaor} and the proposition is achieved.
% 
% To prove the claim, let $\gamma^\la:[0,1]\to\R^2$ be a family with the required properties. 
% Then, 
% Integrando,  usando $\gamma_1(1)=0$  e integrando per parti la $\dot\gamma_1$ si trova
% \[
% \gamma_4(1)= \frac{1}{2}\int_0^1\gamma_1(t)^2 dt=x_4
% \]
% Attenzione: si raggiungono solo punti  con $x_4>0$!!!  La scelta $\gamma_2(t)=t$ \`e restrittiva.
% 
Then 
\begin{equation}\label{hiolk} 
\begin{aligned}
 \abs{2\la} &=  \Bigl|\int_{\gamma^\la}
 x_1^2 dx_2\Bigr| =
 \Bigl|\int_0^1 \g_1^\la(t)^2 \dot\gamma_2^\la (t) dt \Bigr|\leq \sup_{t\in[0,1]}\gamma_1^\la(t)^2 \int_0^1\abs{\dot\gamma_2^\la(t)}
 dt
\\& \leq \sup_{(x_1,x_2) \in \g^\la([0,1])}x_1^2\cdot \operatorname{length}(\gamma^\la)
 \\&\leq 2\sup_{x\in \g^\la}  x_1^2 ,
\end{aligned}
\end{equation}  
where we assumed without loss of generality that $\operatorname{length}(\gamma^\la)\leq 2$ for all  $\abs{\la}$ sufficiently small.
Therefore, there is $  t_\lambda \in(0,1)$ such that $\gamma_1^\la(  t_\la)
=\abs{\la}^{1/2}$.
Thus
\[
\begin{aligned}
 \length(\gamma_\la) & =\length(\gamma|_{[0,  t_\la]})+\length(\gamma|_{[  t_\la,1]})
\\ &
  \geq \abs{(\abs{\la}^{1/2},\gamma_2^\la(t_\la))- (0,0)} +\abs{(\abs{\la}^{1/2},\gamma_2^\la(t_\la))- (0,1)}
  \\&
   \geq \Bigl| \Bigl( \abs{\la}^{1/2},\frac 12\Bigr)-(0,0)\Bigr| +
   \Bigl| \Bigl( \abs{\la}^{1/2},\frac 12\Bigr)-(0,1)\Bigr|
% \\&
  =2\sqrt{\frac 14+\abs{\la}},
\end{aligned}
\]
and the claim follows.
% which is incompatible with In conclusione abbiamo provato che esistono $c_1$ e $c_2$ tali che
% \begin{equation}
%  d(0,1,0,\la)\geq 1+c_1\abs{\la}\quad\text{per ogni $\la$ con } \;\abs{\la}\leq c_2
% \end{equation} 
\end{proof}

\begin{remark}\label{martino1} 
If we let $x_3=0$ and we identify  respectively $(x_1,x_2 ,x_4)$ with $(y,x,z)\in\R^3$, an inspection of the proof above shows that  we have proved the following estimate for the Martinet vector fields $X =\p_x+\frac{y^2}{2}\p_z$ and  $Y=\p_y$, 
\[
\liminf_{z\to 0}\frac{d(1,0, z)-d(1,0,0)}{|z| }\gvertneqq        
0.\]
If $z>0$, then this estimate is contained in   \cite[eq.~(4.31)]{AgrachevBonnard97}, where it is shown that the intersection of the unit sphere with the abnormal set $y=0$,  has a parametrization of the form 
\[
 x(t)=1-t+o(t)\quad\text{and} 
\quad  z(t)=\frac{2}{3\pi^2}t+o(t)\quad\text{as $t\to 0+$}
\]
 If $z<0$, then the absolute value in the 
 first equality in the chain of estimates \eqref{hiolk} is very 
 rough and our argument does not detect the logarithmic estimate proved by \cite{AgrachevBonnard97}.
\end{remark}

\subsection{Failure of horizontal semiconcavity at abnormal points} Here we prove the horizontal 
estimate \eqref{numerodos}.
% the following theorem. 
% \begin{theorem}\label{greg} 
%  We have for all  $x_2\in\R$,
%  \begin{equation*}
%   \limsup_{(y_1,y_2)\to 0}\frac{d\Bigl(e^{y_1X_1+y_2 X_2}(0,x_2,0,0) \Bigr)
%   -d\bigl(0,x_2,0,0 \bigr)}{y_1^2+y_2^2}=+\infty.
%  \end{equation*}
% %  \begin{equation*}
% %   \limsup_{(y_1,y_2)\to 0}\frac{d\bigl(x_2 e_2\cdot (y_1 e_1+y_2 e_2)\bigr)
% %   +d\bigl(
% %   x_2 e_2\cdot (-y_1 e_1-y_2 e_2)
% %   \bigr)-2d\bigl(x_2 e_2\bigr)}{y_1^2+y_2^2}=+\infty
% %  \end{equation*} 
% \end{theorem}
\begin{remark}\label{martino2} 
 An inspection of the proof below shows that no information on the variable $x_3$ is used (we will not make any use of the first equation of \eqref{giroabc}). Thus we get some more information on the distance for the Martinet vector fields $Y=\p_y$ and $X=\p_x+\frac{y^2}{2}\p_z$. Namely we have the estimate  
 \[
  \lim_{y\to 0}\frac{d(1,y,0)-d(1,0,0)}{ y^2}=+\infty 
 \]
\end{remark}

\begin{proof}[Proof of  \eqref{numerodos}]
Since the case $x_2=0$ is trivial,  without loss of generality it suffices to show the statement with $x_2=1$ and $ y_2 =0$. In such case we are able to prove that
 \[
  \lim_{\la\to 0}\frac{    d(e_2 \cdot \la e_1 )  -d(e_2) }
  {\la^2}=+\infty.
 \]
Note that $e_2 \cdot \la e_1=e^{\la X_1}(0,1,0,0)=(\la,1,0,0)$. 
% 
% \[
%  e^{\la X_1}(0,1,0,0)=\Bigl(\la,1,-\frac{\la}{2},-\frac{\la^2}{6}\Bigr),
% \]
Any admissible curve in $\wt \gamma=:(\gamma,\gamma_3,\gamma_4):[0,T]\mapsto \E$ is the lifting of its   plane  projection $\gamma:[0,T]\to \R^2$ with the constraints
$
 \dot\gamma_3=\gamma_1\dot\gamma_2\quad\text{and}\quad \dot\gamma_4=\frac 12\gamma_1^2\dot\gamma_2.
$
Thus, the requirements $\gamma_3(T) =\gamma_4(T)=0$ 
can be written in the form 
% $
%  \gamma (0)=(0,0)$ and $\gamma (1)= (\lambda,1)
% $ 
% , then
% Inportante: vogliano che la lunghezza di $\gamma$ non ecceda $1+C\lambda^2$. La componente $\gamma_3$ e $\gamma_4$ si lifta secondo la formula $\dot\gamma_3=\frac 12\gamma_1\dot\gamma_2
% -
% \frac 12 \dot\gamma_2\gamma_1$. QUindi il vincolo su $\gamma_3(1)$ diventa:
\begin{equation}\label{giroabc} 
% \label{priss} 
\begin{aligned}
 \int_\gamma x_1 dx_2   
 = 0                                    
\quad\text{and }
\quad % \\
\int_\gamma x_1^2 dx_2=0
% - x_1 x_2 dx_1 
% &
% =  -\frac 12 \la^2. 
\end{aligned}
 \end{equation}  
(the first equality will not be used in our argument).

Let us assume by contradiction that there exists a family of curves 
$x^\la :[0,T_\la]\to\R^2$ with $\la\in\R$ and a constant $C_0>0$ 
such that for all $\lambda$ close to $0$
all the following properties hold:
\begin{equation}
\left\{\begin{aligned}
 &x^\la(0)=(0,0),\qquad
%  \\&
 x^\la(T_\la)=(\la,1)
 \\&
 \abs{\dot x^\la}\leq 1\quad\text{a.e.}
 \\& T_\la-1\leq C_0 \la^2
 \\&
 \int_{0}^{T_\la}  x_1^\la(t)^2\dot x_2^\la (t) dt=0
                \end{aligned}
                \right. 
 \end{equation}
We will show that this produces the following contradiction. Letting
\begin{equation*}
\text{(LHS)}:= \int_{\{\dot x_2^\la>0\}}(x_1^\la)^2\dot x_2^\la
=\Bigl|\int_{\{\dot x_2^\la\leq 0\}} (x_1^\la)^2 \dot x_2^\la\Bigr|=:\text{(RHS)}.
\end{equation*}
we claim that there are $C_1, C_2>0$ such that if $\abs{\lambda}$ is sufficiently small, then 
\begin{equation}
 \text{(RHS)}\leq C_1\la^4\quad \text{and}\qquad  \text{(LHS)}\geq  C_2\abs{\la}^3. 
\end{equation} 
 
 To get this contradiction, by symmetry it suffices to discuss the case $\la>0$. 
 The proof is articulated in several steps. 
 
\step{Step 1.} First we prove the estimate $|\{t\in[0, T_\la]: \dot x_2^\la(t)\leq 0\}|\leq C_0\la^2$.
This can be achieved easily because
\[
 1=x_2^\la(T)-x_2^\la(0)=\int_0^{T_\la}\dot x_2^\la =\int_{\dot x_2^\la >0}\dot x_2^\la
 +\int_{\dot x_2^\la<0}\dot x_2^\la
 \leq \int_{\dot x_2^\la>0}\dot x_2^\la 
 \leq |\{\dot x_2^\la>0 \}|,
\]
because  $\abs{\dot x_2^\la}\leq \abs{\dot x^\la}\leq 1$. Thus 
$|\{\dot x_2^\la> 0\}|\geq 1$.
Therefore, its complementary set satisfies  
\[
 |\{\dot x_2^\la\leq 0\}|=T_\la-|\{\dot x_2^\la> 0\}|\leq 1+C_0\la^2-1=C_0\la^2.
\]

\step{Step 2.} There is $C_1>0$ such that 
$\sup_{[0,T_\la]} (x_1^\la)^2\leq C_1\la^2.
$

Let $\la>0$  and define the positive number $q_\la=\max_{t\in[0, T_\la]} x_1^\la(t)/\la $.  (An analogous discussion, left to the reader, can be given working with $p_\la:= \min _{t\in[0, T_\la]} x_1^\la(t)/\la$). Assume that $q_\la\geq 2$, otherwise there is nothing to prove.  
Take a point $( \lambda q_\la ,  x_2^\la) \in\gamma^\la([0,T_\la])$. Then
\[
\begin{aligned}
 1+C_0\la^2&\geq \length(x^\la)
 \geq d((\la q_\la,x_2^\la),(\la,1))+d((\la q_\la,x_2^\la),(0,0))
 \\&\geq d((\la q_\la,x_2^\la),(\la,1))+d((\la q_\la,x_2^\la),(\lambda,0))
 \\&\geq
 \text{(this quantity is minimal for  $x_2^\la=\frac 12$)}
\\&\geq  d((\la q_\la,1/2),(\la,1))+d((\la q_\la,1/2),(\lambda,0))
\\&=2\sqrt{\frac 14+(q_\la -1)^2\la^2}.
 \end{aligned}
\]
Comparing the first and the last term, we see that $q_\la$ should be bounded  uniformly for small 
positive $ \la $.

\step{Step 3.} Estimate of  (RHS):
\[
\Bigl| \int_{\dot x_2^\la<0}(x_1^\la)^2\dot x_2^\la\Bigr|\leq \sup_{[0,T_\la]} (x_1^\la)^2 
\abs{\{\dot x_2^\la<0\}}\leq C_1\la^2 \cdot C_0\la^2,
\]
as desired.

The estimate of (LHS) is more delicate and we need some  preliminary notation.
Introduce the following rotation $ \rho_\la: \R^2\to\R^2$
% che porta il punto $(0,\sqrt{1+\la^2})$ nel punto $(\la,1)$.
\begin{equation}\label{matrice} 
 \r_\la
 \begin{pmatrix}
   \xi_1\\ \xi_2
  \end{pmatrix}=
  \frac{1}{\sqrt{1+\la^2}}
  \begin{pmatrix}
   1& \la
   \\
   -\la & 1
  \end{pmatrix}
%   \begin{psmallmatrix}
% \frac{1}{\sqrt{1+\la^2}}  & \frac{\la}{\sqrt{1+\la^2}}
% \\ \frac{-\la}{\sqrt{1+\la^2}} & \frac{1}{\sqrt{1+\la^2}} \end{psmallmatrix}
\begin{pmatrix}
   \xi_1\\\xi_2
  \end{pmatrix}% \begin{pmatrix}
% \xi \\ \eta  =
% \frac{1}{\sqrt{1+\la^2}}=
=\begin{pmatrix}
(\xi_1+\la\xi_2)/\sqrt{1+\la^2}
\\
(\xi_2 -\la \xi_1)/\sqrt{1+\la^2}
\end{pmatrix}.
\end{equation} 
Observe that  $\rho_\la
(0,\sqrt{1+\la^2})=(\la,1)$ for all $\la$.
Define then 
\begin{equation}\label{pizzero} 
 p_0=2\sqrt{2C_0\;}  
\end{equation} 
and construct the following sets (we will work both with these sets and with their rotated through $\r_\la$). 
\begin{equation*}
 \wt \ell_\la:=\{(\xi_1,\xi_2):\xi_2=\sqrt{1+\la^2}(1-\la)\}.
\end{equation*} 
This is a horizontal line below the point $(0,\sqrt{1+\la^2})$ of an amount   of order $\la$. Inside this line we fix the (rather short) segment 
\begin{equation*}\begin{aligned}
 \wt F_\la  & =\{(\xi_1,\xi_2):\xi_2=\sqrt{1+\la^2}(1-\la),\abs{\xi_1}\leq p_0\la^{3/2}\}
 \\&
 =\Bigl\{\bigl (\theta p_0\la^{3/2},\sqrt{1+\la^2}(1-\la)\bigr ): \abs{\theta}\leq 1\Bigr\}
 \end{aligned}
\end{equation*} 
and the tiny rectangle
\[
\wt R_\la 
:=  \Bigl\{\bigl (\theta_1 p_0\la^{3/2},\sqrt{1+\la^2}(1-\theta_2 \la)\bigr ): \abs{\theta_1}\leq 1,\quad\abs{\theta_2}\leq 1\Bigr\}. 
\]
which extends on top of  $\wt F_\la$ of an amount  approximately $2\lambda$.
Then, on the left and on the right of $\wt R_\la$ introduce the set 
\begin{equation*}
\wt M_\la :=  \Bigl\{\bigl (\theta_1 p_0\la^{3/2},\sqrt{1+\la^2}(1-\theta_2 \la)\bigr ): \abs{\theta_1}\geq 1,\quad\abs{\theta_2}\leq 1\Bigr\}. 
\end{equation*} 
Finally, on top of $\wt R_\la\bigcup \wt M_\la$ we have the half-plane
\[
\wt G_\la :=  \Bigl\{\bigl (\theta_1 p_0\la^{3/2},\sqrt{1+\la^2}(1-\theta_2 \la)\bigr ):  \theta_1 \in\R ,\quad \theta_2 \leq -1\Bigr\}. 
\]
Correspondingly we have the rotated sets 
$
 \ell_\la:= \rho_\la\wt \ell_\la$, $F_\la:=\rho_\la\wt F_\la$, $R_\la:=\rho_\la\wt R_\la
 ,$ and  $ M_\la=\rho_\la\wt M_\la
$. 
The tiny rectangle  $R_\la$ is centered at the final point  $(\la,1)$.

\step{Step 4.}  Under the choice of $p_0$ made in \eqref{pizzero}, we have for sufficiently small   positive $\la$
\begin{equation*}
 x^\la([0,T_\la])\cap M_\la =\varnothing\qquad \text{and} 
\quad x^\la([0,T_\la])\cap G_\la=\varnothing .
 \end{equation*} 
% e
% \begin{equation}\label{laseconda} 
% .\end{equation} 

We start with the proof of the first claim, which gives the more striking information, due to the power $\la^{3/2}$  in the horizontal size of $R_\la$. 
We work with the rotated curve  $\xi^\la(t)= \rho_\la^{-1} x^\la(t)$. Such curve has length at most 
$ 1+C_0\la^2$ and connects $(0,0)$ with  $(0,\sqrt{1+\la^2}).$
Assume by contradiction that there is a point  
% of $\xi^\la[0, T_\la]$
belonging to  $\wt M_\la\cap\xi^\la([0,T_\la])$.
Such point has the form  $\bigl(\theta_1 p_0\la^{3/2},
\sqrt{1+\la^2} (1-\theta_2\la)\bigr)$, for some $\theta_1,\theta_2$ satisfying
$
\abs{\theta_1}\geq 1$, and  $\abs{\theta_2}\leq 1$. Therefore, the estimate on the length furnishes
\begin{equation*}
\begin{aligned}
 1+C_0\la^2  &\geq \length(x^\la)
 \\& \geq d\Bigl ((0,0), (\theta_1 p_0\la^{3/2}, \sqrt{1+\la^2} (1-\theta_2\la))\Bigr )
 \\&
 \qquad +
 d\Bigl ((\theta_1 p_0\la^{3/2}, 
 \sqrt{1+\la^2} (1-\theta_2\la)
 ),(0,\sqrt{1+\la^2})\Bigr)
 \\&\geq \text{(we minimize choosing $\theta_2=1$)}
%  \\&
%  \geq 
%  d\Bigl ((0,0), (p\la^{3/2}, \sqrt{1+\la^2} (1- \la))\Bigr )+
%  d\Bigl ((p\la^{3/2}, 
%  \sqrt{1+\la^2} (1-\la)
%  ),(0,\sqrt{1+\la^2}\Bigr)
 \\&
 \geq\sqrt{\theta_1^2p_0^2\la^3+(1+\la^2)(1-\la)^2}
 +\sqrt{\theta_1^2p_0^2\la^3+(1+\la^2)\la^2}
 \\&
 \geq 1-\la 
 +\sqrt{\theta_1^2p_0^2\la^3+(1+\la^2)\la^2}
% \\&
 \geq 1-\la +\la\sqrt{1+p_0^2\la}.
 \end{aligned}
\end{equation*}
Comparing the first and the last term, we see that this chain of inequality conflicts with the choice of $p_0$ made in \eqref{pizzero}, for small $\lambda$.

 Next we show the second statement of \emph{Step 4.}
Let $\la>0$ be a small number and assume by contradiction that there exists $\bar x^\la
\in G_\la\cap x^\la([0,T_\lambda])$. The rotated point  
$ \ol \xi ^\la:=\rho_\la^{-1} \ol x^\la$ has the form 
$(\theta_1p_0\la^{3/2}, \sqrt{1+\la^2} (1-\theta_2\la))$ with $\theta_1\in\R$ and $ 
\theta_2\leq-1$. Thus, 
 it must be $\bar \xi_2^\la\geq (1+\la)\sqrt{1+\la^2}$. Therefore 
\[
\begin{aligned}
 1+C_0\la^2 &\geq d((0,0),(\bar \xi_1^\la,\bar \xi_2^\la))+d((\bar \xi_1^\la,\bar \xi_2^\la)
  ,(0,\sqrt{1+\la^2}))
 \\&
 \geq
 \abs{\bar \xi_2^\la}+\abs{\bar \xi_2^\la -
 \sqrt{1+\la^2}}
 \\&
 \geq (1+\la)\sqrt{1+\la^2}+\bigl((1+\la)\sqrt{1+\la^2}-\sqrt{1+\la^2}\bigr)
 \\&=(1+2\la)\sqrt{1+\la^2}.
\end{aligned}
\]
Again, comparing the first and last term, we find a contradiction and \emph{Step 4.} 
is accomplished.

\step{Step 5.} We claim that if  $(  x_1 ,  x_2 )\in F_\la$, then 
$
  x_2^\la\leq 1-\frac{\la}{2}.$
Here we use the fact that the segment $F_\la$ is very short with respect to $\la$.

To check the claim, recall that 
$  x\in F_\la$  means that  there is $ \theta\in[-1,1] $ 
such that 
$
 (    x_1 
  ,  x_2)=\rho_\la(
\theta  p_0\la^{3/2},
\sqrt{1+\la^2}(1-\la))$
Thus, using \eqref{matrice}, we find 
\[
 \bar x_2=\frac{1}{\sqrt{1+\la^2}}\Bigl(-\theta p_0\la^{5/2}+\sqrt{1+\la^2}(1-\la)\Bigr)
\]
and Step 5 is accomplished, if $\la>0$ is sufficiently small.

\step{Step 6.} If $  x=(x_1,x_2)\in R_\la$, then we have 
% \begin{equation*}
$ x_1\geq \frac{\la}{2}.$
% \end{equation*}

This can be seen again by means of  \eqref{matrice}, which gives for suitable 
$ \theta_1, \theta_2\in[-1,1] $
\[
   x_1=\frac{1}{\sqrt{1+\la^2}}\Bigl(\theta_1 p_0\la^{3/2}+
 \la\sqrt{1+\la^2}(1-\theta_2\la)\Bigr)
 \geq  \frac{\la}{2},
\]
for all positive $\lambda $ sufficiently small.

\step{Step 7.} Lower estimate of (LHS).

Take $\lambda$ and the corresponding curve $x^\la$. Let $  t_\la\in[0,T_\la]$ be the unique time such that 
\[
 x^\la ( t_\la) \in F_\la\qquad x^\la(t ) \notin F_\la\quad\forall t\in\mathopen]  t_\la,T_\la].
\]
Note that $x^\la (t)
\in R_\la$ 
for all   $t\in[ t_\la,T_\la]$.
This follws fron the fact that the curve  $x^\la$ can intersect the line $\ell_\la$ only in the segment $F_\la$. Thus, after the time $t_\la$ it should lie   on top of  such line. On the other side, by \emph{Step 4.}, the curve cannot touch the ``prohibited set'' 
$M_\la\cup G_\la$. Therefore  $x^\la([t_\la,T_\la])\subset R_\la$. Therefore
\[
\begin{aligned}
 \int_{\dot x_2^\la>0}(x_1^\la)^2 \dot x_2^\la
& \geq \int_{[  t_\la,T_\la]\cap\{ \dot x_2^\la>0\}} (x_1^\la)^2 \dot x_2^\la 
 \geq\inf_{(x_1,x_2)\in R_\la} x_1 ^2  \int_{[  t_\la,T_\la]
 \cap\{ \dot x_2^\la>0\}} \dot x_2^\la 
 \geq \inf_{(x_1,x_2)\in R_\la} x_1 ^2  \int_{[ t_\la,T_\la]} \dot x_2^\la 
 \\&\geq\text{(By  Step 6)} 
 \geq
 \frac{\la^2}{4} (x_2^\la(T_\la)-x_2^\la( t_\la))
\\&\geq
 \text{(By Step  5) }\geq  \frac{\la^2}{4} \Bigl(1-\Bigl(1-\frac{\la}{2}\Bigr)\Bigr)
 =\frac{\la^3}{8},
\end{aligned}
\]
and the proof is concluded. 
\end{proof}

\section*{Acknowledgements}  
The authors are members of the {\it Gruppo Nazionale per
l'Analisi Matematica, la Probabilit\`a e le loro Applicazioni} (GNAMPA)
of the {\it Istituto Nazionale di Alta Matematica} (INdAM)

\footnotesize

 \phantomsection
\addcontentsline{toc}{section}{References}

   \newcommand{\etalchar}[1]{$^{#1}$}
\def\cprime{$'$} \def\cprime{$'$}
\providecommand{\bysame}{\leavevmode\hbox to3em{\hrulefill}\thinspace}
\providecommand{\MR}{\relax\ifhmode\unskip\space\fi MR }
\providecommand{\MRhref}[2]{%
  \href{http://www.ams.org/mathscinet-getitem?mr=#1}{#2}
}
\providecommand{\href}[2]{#2}


\begin{thebibliography}{LDLMV13}

\bibitem[ABB16]{AgrachevBarilariBoscain}
A.~{Agrachev}, D.~{Barilari}, and U.~{Boscain}, \emph{{Notes on Riemannian and
  SubRiemanniann geometry}}, preprint (2016).

\bibitem[ABCK97]{AgrachevBonnard97}
A.~Agrachev, B.~Bonnard, M.~Chyba, and I.~Kupka, \emph{Sub-{R}iemannian sphere
  in {M}artinet flat case}, ESAIM Control Optim. Calc. Var. \textbf{2} (1997),
  377--448 (electronic).

\bibitem[AGL15]{AgrachevGentileLerario}
Andrei~A. Agrachev, Alessandro Gentile, and Antonio Lerario, \emph{Geodesics
  and horizontal-path spaces in {C}arnot groups}, Geom. Topol. \textbf{19}
  (2015), no.~3, 1569--1630.

\bibitem[{Agr}15]{Agrachev15}
A.~{Agrachev}, \emph{{Tangent hyperplanes to subriemannian balls}}, ArXiv
  e-prints (2015).

\bibitem[AS04]{Agrachev}
Andrei~A. Agrachev and Yuri~L. Sachkov, \emph{Control theory from the geometric
  viewpoint}, Encyclopaedia of Mathematical Sciences, vol.~87, Springer-Verlag,
  Berlin, 2004, Control Theory and Optimization, II.

\bibitem[AS11]{ArdentovSachkov11}
A.~A. Ardentov and Yu.~L. Sachkov, \emph{Extremal trajectories in the nilpotent
  sub-{R}iemannian problem on the {E}ngel group}, Mat. Sb. \textbf{202} (2011),
  no.~11, 31--54.

\bibitem[AS15]{ArdentovSachkov14}
\bysame, \emph{Cut time in sub-{R}iemannian problem on {E}ngel group}, ESAIM
  Control Optim. Calc. Var. \textbf{21} (2015), no.~4, 958--988. \MR{3395751}

\bibitem[AT13]{AdamsTie13}
Malcolm~R. Adams and Jingzhi Tie, \emph{On sub-{R}iemannian geodesics on the
  {E}ngel groups: {H}amilton's equations}, Math. Nachr. \textbf{286} (2013),
  no.~14-15, 1381--1406.

\bibitem[BLU07]{BonfiglioliLanconelliUguzzoni}
A.~Bonfiglioli, E.~Lanconelli, and F.~Uguzzoni, \emph{Stratified {L}ie groups
  and potential theory for their sub-{L}aplacians}, Springer Monographs in
  Mathematics, Springer, Berlin, 2007.

\bibitem[CJT06]{ChitourJeanTrelat}
Y.~Chitour, F.~Jean, and E.~Tr{\'e}lat, \emph{Genericity results for singular
  curves}, J. Differential Geom. \textbf{73} (2006), no.~1, 45--73.

\bibitem[CR08]{CannarsaRifford}
P.~Cannarsa and L.~Rifford, \emph{Semiconcavity results for optimal control
  problems admitting no singular minimizing controls}, Ann. Inst. H. Poincar\'e
  Anal. Non Lin\'eaire \textbf{25} (2008), no.~4, 773--802.

\bibitem[CS04]{CannarsaSinestrari}
Piermarco Cannarsa and Carlo Sinestrari, \emph{Semiconcave functions,
  {H}amilton-{J}acobi equations, and optimal control}, Progress in Nonlinear
  Differential Equations and their Applications, 58, Birkh\"auser Boston, Inc.,
  Boston, MA, 2004.

\bibitem[DMO{\etalchar{+}}15]{LeDonneMontgomeryOttazziPansuVittone}
Enrico~Le Donne, Richard Montgomery, Alessandro Ottazzi, Pierre Pansu, and
  Davide Vittone, \emph{Sard property for the endpoint map on some carnot
  groups}, Annales de l'Institut Henri Poincare (C) Non Linear Analysis (2015),
  --.

\bibitem[FR10]{FigalliRifford10}
Alessio Figalli and Ludovic Rifford, \emph{Mass transportation on
  sub-{R}iemannian manifolds}, Geom. Funct. Anal. \textbf{20} (2010), no.~1,
  124--159.

\bibitem[GT11]{GutierrezTournier11}
Cristian~E. Guti{\'e}rrez and Federico Tournier, \emph{Harnack inequality for a
  degenerate elliptic equation}, Comm. Partial Differential Equations
  \textbf{36} (2011), no.~12, 2103--2116. \MR{2852071}

\bibitem[Hsu92]{hsu}
Lucas Hsu, \emph{Calculus of variations via the {G}riffiths formalism}, J.
  Differential Geom. \textbf{36} (1992), no.~3, 551--589.

\bibitem[Kis03]{Kishimoto}
Iwao Kishimoto, \emph{Geodesics and isometries of {C}arnot groups}, J. Math.
  Kyoto Univ. \textbf{43} (2003), no.~3, 509--522.

\bibitem[LDLMV13]{LeDonneLeonardiMontiVittone}
Enrico Le~Donne, Gian~Paolo Leonardi, Roberto Monti, and Davide Vittone,
  \emph{Extremal curves in nilpotent {L}ie groups}, Geom. Funct. Anal.
  \textbf{23} (2013), no.~4, 1371--1401.

\bibitem[LS95]{LiuSussmann}
Wensheng Liu and H{\'e}ctor~J. Sussman, \emph{Shortest paths for
  sub-{R}iemannian metrics on rank-two distributions}, Mem. Amer. Math. Soc.
  \textbf{118} (1995), no.~564, x+104.

\bibitem[M{\'e}t80]{Metivier80}
Guy M{\'e}tivier, \emph{Hypoellipticit\'e analytique sur des groupes nilpotents
  de rang {$2$}}, Duke Math. J. \textbf{47} (1980), no.~1, 195--221.

\bibitem[Mon02]{Montgomery}
Richard Montgomery, \emph{A tour of subriemannian geometries, their geodesics
  and applications}, Mathematical Surveys and Monographs, vol.~91, American
  Mathematical Society, Providence, RI, 2002.

\bibitem[Mon14]{Montanari14}
Annamaria Montanari, \emph{Harnack inequality for a subelliptic pde in
  nondivergence form}, Nonlinear Analysis: Theory, Methods \& Applications
  \textbf{109} (2014), 285--300.

\bibitem[MS04]{MullerSeeger04}
Detlef M{\"u}ller and Andreas Seeger, \emph{Singular spherical maximal
  operators on a class of two step nilpotent {L}ie groups}, Israel J. Math.
  \textbf{141} (2004), 315--340.

\bibitem[{Rif}14]{Rifford14}
Ludovic {Rifford}, \emph{{Sub-Riemannian geometry and optimal transport.}},
  Cham: Springer; Bilbao: BCAM -- Basque Center for Applied Mathematics, 2014
  (English).

\bibitem[Sus96]{Sussmann96}
H{\'e}ctor~J. Sussmann, \emph{A cornucopia of four-dimensional abnormal
  sub-{R}iemannian minimizers}, Sub-{R}iemannian geometry, Progr. Math., vol.
  144, Birkh\"auser, Basel, 1996, pp.~341--364.

\bibitem[Tra12]{Tralli12}
Giulio Tralli, \emph{Double ball property for non-divergence horizontally
  elliptic operators on step two {C}arnot groups}, Atti Accad. Naz. Lincei Cl.
  Sci. Fis. Mat. Natur. Rend. Lincei (9) Mat. Appl. \textbf{23} (2012), no.~4,
  351--360. \MR{2999549}

\bibitem[Tr{\'e}00]{Trelat}
E.~Tr{\'e}lat, \emph{Some properties of the value function and its level sets
  for affine control systems with quadratic cost}, J. Dynam. Control Systems
  \textbf{6} (2000), no.~4, 511--541.

\end{thebibliography}
\end{document}